\documentclass[leqno,a4paper,12pt]{amsart}
\usepackage{fullpage}
\usepackage[english]{babel}
\usepackage[utf8]{inputenc}
\usepackage{amsmath,amssymb,amsthm,stmaryrd,appendix}
\usepackage{eqparbox}
\usepackage{graphicx,xcolor}
\usepackage{float}
\usepackage{etoolbox}
\patchcmd{\thmhead}{(#3)}{#3}{}{}

\usepackage[all,2cell]{xy}

\usepackage{datetime}

\usepackage{enumitem}
\setlist[enumerate]{label=\mbox{(\textit{\roman*}\hspace{.08em})},font=\rm,itemsep=.25em}
\setlist[itemize]{itemsep=.25em}

\newcommand{\hair}{\ifmmode\mskip1.5mu\else\kern0.08em\fi}

\usepackage[colorlinks]{hyperref}
\hypersetup{
  pdftitle   = {Deformations of categories of coherent sheaves via quivers with relations},
  pdfauthor  = {Severin Barmeier, Zhengfang Wang},
  pdfcreator = {},
  linkcolor  = {red},
  citecolor  = {red},
  urlcolor   = {red}
}
\usepackage{tikz}
\usepackage{mathtools}
\usetikzlibrary{backgrounds,matrix,arrows,decorations.markings,decorations.pathmorphing,shapes}
\tikzset{>=stealth}
\newcommand{\tikzto}{\mathrel{\tikz[baseline]\draw[ ->,line width=.4pt] (0ex,0.65ex) -- (3ex,0.65ex);}}
\newcommand{\tikzmapsto}{\mathrel{\tikz[baseline]\draw[|->,line width=.4pt] (0pt,0.65ex) -- (3ex,0.65ex);}}

\newcommand{\tikzrightrightarrows}{\mathrel{\tikz[baseline]{\draw[->,line width=.4pt] (0ex,0.25ex) -- (3ex,0.25ex); \draw[->] (0ex,1.05ex) -- (3ex,1.05ex);}}}
\newcommand{\tikzlongrightarrow}{\mathrel{\tikz[baseline]\draw[->,line width=.4pt] (0ex,0.65ex) -- (4ex,0.65ex);}}
\newcommand{\tikzlongleftarrow}{\mathrel{\tikz[baseline]\draw[<-,line width=.4pt] (0ex,0.65ex) -- (4ex,0.65ex);}}
\newcommand{\tikzlongleftrightarrow}{\mathrel{\tikz[baseline]\draw[<->,line width=.4pt] (0ex,0.65ex) -- (4ex,0.65ex);}}
\newcommand{\tikzhookrightarrow}{\mathrel{\tikz[baseline]\draw[right hook->,line width=.4pt] (0ex,0.65ex) -- (3ex,0.65ex);}}
\newcommand{\tikzhookleftarrow}{\mathrel{\tikz[baseline]\draw[<-left hook,line width=.4pt] (0ex,0.65ex) -- (3ex,0.65ex);}}

\newcommand{\toarg}[1]{\mathrel{\tikz[baseline]\path[->,line width=.4pt] (0ex,0.65ex) edge node[above=-.4ex, overlay, font=\scriptsize] {$#1$} (3.5ex,.65ex);}}

\newcommand{\smallto}{\mathrel{\tikz[baseline]\path[->,line width=.3pt] (0ex,0.5ex) edge (1.5ex,.5ex);}}
\newcommand{\toarglong}[1]{\mathrel{\tikz[baseline]\path[->,line width=.4pt] (0ex,0.65ex) edge node[above=-.4ex, overlay, font=\scriptsize] {$#1$} (5ex,.65ex);}}
\newcommand{\leftarrowarg}[1]{\mathrel{\tikz[baseline]\path[<-,line width=.4pt] (0ex,0.65ex) edge node[above=-.4ex, overlay, font=\scriptsize,pos=.6] {$#1$} (3.5ex,.65ex);}}
\newcommand{\longleftarrowarg}[1]{\mathrel{\tikz[baseline]\path[<-,line width=.4pt] (0ex,0.65ex) edge node[above=-.4ex, overlay, font=\scriptsize,pos=.6] {$#1$} (5ex,.65ex);}}

\newcommand{\tosim}{\mathrel{\tikz[baseline] \path[->,line width=.4pt] (0ex,0.65ex) edge node[above=-.9ex, overlay, font=\normalsize, pos=.45] {$\sim$} (3.5ex,.65ex);}}

\newcommand{\hdot}{\bullet}
\newcommand{\ldot}{\bullet}

\newcommand{\llrr}[1]{%
\llbracket #1 \rrbracket}

\newcommand{\hatotimes}{\mathbin{\widehat{\otimes}}}

\newcommand{\blank}{\mkern.5mu {-} \mkern.5mu}

\renewcommand{\theta}{\vartheta}
\renewcommand{\phi}{\varphi}
\renewcommand{\emptyset}{\varnothing}

\numberwithin{equation}{section}
\newtheorem{theorem}[equation]{Theorem}
\newtheorem*{theorem*}{Theorem}
\newtheorem{proposition}[equation]{Proposition}

\newtheorem{question}[equation]{Question}
\newtheorem{lemma}[equation]{Lemma}
\newtheorem*{lemma*}{Lemma}
\newtheorem{corollary}[equation]{Corollary}
\newtheorem*{corollary*}{Corollary}
\theoremstyle{definition}
\newtheorem{definition}[equation]{Definition}
\newtheorem{example}[equation]{Example}

\theoremstyle{remark}
\newtheorem{remark}[equation]{Remark}

\patchcmd{\thmhead}{(#3)}{#3}{}{} 

\renewcommand{\Bar}{\operatorname{Bar}}
\DeclareMathOperator{\coh}{coh}

\DeclareMathOperator{\D}{D}

\newcommand{\e}{\mathrm{e}}
\DeclareMathOperator{\Ext}{Ext}
\DeclareMathOperator{\End}{End}
\DeclareMathOperator{\GL}{GL}

\renewcommand{\H}{\operatorname{H}}
\DeclareMathOperator{\Har}{Har}
\DeclareMathOperator{\HH}{HH}
\DeclareMathOperator{\Hom}{Hom}
\newcommand{\Irr}{\mathrm{Irr}}

\DeclareMathOperator{\Mod}{Mod}

\DeclareMathOperator{\im}{im}

\DeclareMathOperator{\Qcoh}{Qcoh}

\DeclareMathOperator{\SL}{SL}
\DeclareMathOperator{\Spec}{Spec}

\DeclareMathOperator{\Tot}{Tot}
\newcommand{\U}{\mathrm{U}}

\DeclareFontFamily{U}{mathc}{}
\DeclareFontShape{U}{mathc}{m}{it}%
{<->s*[1.01] mathc10}{}
\DeclareMathAlphabet{\mathcal}{U}{mathc}{m}{it}

\begin{document}

\title[Deformations of categories of coherent sheaves via quivers]{Deformations of categories of coherent sheaves via quivers with relations}

\author{Severin Barmeier}
\email{s.barmeier@gmail.com}
\address{Universität zu Köln, Mathematisches Institut, Weyertal 86-90, 50931 Köln, Germany \textit{and} Max-Planck-Institut für Mathematik, Vivatsgasse 7, 53111 Bonn, Germany}

\author{Zhengfang Wang}
\email{wangzg@mathematik.uni-stuttgart.de}
\address{Universität Stuttgart, Institut für Algebra und Zahlentheorie, Pfaffenwaldring 57, 70569 Stuttgart, Germany \textit{and} Max-Planck-Institut für Mathematik, Vivatsgasse 7, 53111 Bonn, Germany}



\begin{abstract}
We give an explicit description of the deformation theory of the Abelian category of (quasi)coherent sheaves on any separated Noetherian scheme $X$ via the deformation theory of path algebras of quivers with relations, by using any affine open cover of $X$, or any tilting bundle on $X$, if available.

We also give sufficient criteria for obtaining algebraizations of formal deformations, in which case the deformation parameters can be evaluated to a constant and the deformations can be compared to the original Abelian category on equal terms. We give concrete examples as well as applications to the study of noncommutative deformations of singularities.
\end{abstract}

\maketitle

\tableofcontents

\section{Introduction}

A momentous result due to Gabriel \cite{gabriel} and Rosenberg \cite{rosenberg,brandenburg} shows that any quasi-separated scheme can be reconstructed from its Abelian category of quasi-coherent sheaves. Abelian categories which are ``close to'' categories of quasi-coherent sheaves on a variety or scheme can thus be viewed as generalizations of commutative schemes and are of central importance in noncommutative algebraic geometry.

In \cite{lowenvandenbergh1,lowenvandenbergh2} W.~Lowen and M.~Van den Bergh developed a deformation theory for abstract Abelian categories, controlled by a version of Hochschild cohomology $\H_{\mathrm{Ab}}^\hdot$ for Abelian categories, with first-order deformations parametrized by $\H_{\mathrm{Ab}}^2$ and obstructions lying in $\H_{\mathrm{Ab}}^3$. For a separated Noetherian scheme $X$ over an algebraically closed field $\Bbbk$ of characteristic $0$, Lowen and Van den Bergh showed that the Hochschild cohomology of the Abelian categories
\begin{itemize}
\item $\Mod (\mathcal O_X)$ of all sheaves of $\mathcal O_X$-modules
\item $\Qcoh (X)$ of quasi-coherent sheaves
\item $\coh (X)$ of coherent sheaves
\end{itemize}
are all isomorphic to the Hochschild cohomology $\HH^\hdot (X)$ of the scheme $X$.

For smooth $X$, first-order deformations are thus parametrized by
\[
\HH^2 (X) \simeq \H^0 (\Lambda^2 \mathcal T_X) \oplus \H^1 (\mathcal T_X) \oplus \H^2 (\mathcal O_X)
\]
where the isomorphism is given by the Hochschild--Kostant--Rosenberg decomposition \cite{swan,yekutieli1}. This shows that the deformation theory of $\Mod (\mathcal O_X)$, $\Qcoh (X)$ and $\coh (X)$ can be understood as an amalgamation of quantizations of algebraic Poisson structures on $X$, classical commutative deformations of $X$ and deformations of the gerbe structure of $\mathcal O_X$ (see \S\ref{section:applications}). Further aspects of this deformation theory were studied for example in \cite{benbassatblockpantev,dinhvanhermanslowen,dinhvanliulowen,liulowen}.

Lowen and Van den Bergh showed that for any associative algebra $A$ one has an equivalence of deformation theories
\begin{equation}
\label{eq}
\begin{matrix}
\text{deformations of $A$} \\
\text{as associative algebra}
\end{matrix}
\; \tikzlongleftrightarrow \;
\begin{matrix}
\text{deformations of $\Mod (A)$} \\
\text{as Abelian category}
\end{matrix}
\end{equation}
and in case $A$ is commutative $\Mod (A) \simeq \Qcoh (\Spec A)$, so that this equivalence shows that deformations of $\Qcoh (\Spec A)$ are determined by {\it associative} deformations of $A$. (Note that ``geometric'' deformations of $\Spec A$ correspond to {\it commutative} deformations of $A$, which are trivial if $\Spec A$ is smooth.)

If $X$ is now a quasi-compact separated scheme, then we have an equivalence $\Qcoh (X) \simeq \Qcoh (\mathcal O_X \vert_{\mathfrak U})$, where $\Qcoh (\mathcal O_X \vert_{\mathfrak U})$ is the Abelian category of quasi-coherent modules for the diagram of algebras $\mathcal O_X \vert_{\mathfrak U}$ obtained by restricting the structure sheaf $\mathcal O_X$ to an affine open cover $\mathfrak U$ closed under intersections. Lowen and Van den Bergh showed that in fact the full formal deformation theories of $\Mod (\mathcal O_X)$, $\Qcoh (X)$ and $\coh (X)$ are equivalent to the deformation theory of $\mathcal O_X \vert_{\mathfrak U}$ as a twisted presheaf of associative algebras, which can be viewed as a ``global'' analogue of \eqref{eq} (see Theorem \ref{theorem:equivalence}).

The deformation theory of a diagram of algebras is controlled by an L$_\infty$ algebra structure on the Gerstenhaber--Schack complex of the diagram \cite{gerstenhaberschack}, which can be viewed as a generalization of the classical Hochschild complex which controls the deformation theory of a single algebra --- for a single algebra, the Gerstenhaber--Schack complex coincides with the Hochschild complex. This L$_\infty$ algebra structure controlling formal deformations of diagrams of algebras, or more generally of prestacks, was constructed by H.\ Dinh Van and W.\ Lowen \cite{dinhvanlowen}. Their construction uses (a categorical analogue of) the observation \cite{gerstenhaberschack} that to any diagram $\mathcal A$ of algebras, one can associate a single algebra $\mathcal A!$, whose deformation theory is also equivalent to that of the diagram $\mathcal A$.

In the case of a single algebra, this L$_\infty$ algebra structure recovers the classical DG Lie algebra structure on the Hochschild complex given by the Gerstenhaber bracket. Although these higher structures on the Gerstenhaber--Schack complex give a natural obstruction calculus for deformations of diagrams of algebras, it is not clear how to {\it construct} formal deformations from this point of view. For example, already for the single algebra $A = \Bbbk [x_1, \dotsc, x_d]$, i.e.\ the algebra of global functions on affine space $\mathbb A^d$, the explicit construction of a formal deformation is nontrivial --- formal deformations of $\Qcoh (\mathbb A^d) \simeq \Mod (A)$ are given by deformation quantizations of Poisson structures on $\mathbb A^d$ which were constructed by M.~Kontsevich as part of the proof of his Formality Conjecture \cite{kontsevich1,kontsevich2} and the explicit formula uses certain integrals over configuration spaces of points which are related to multiple zeta values \cite{bankspanzerpym}.

In \cite{barmeierwang} we used the combinatorics of reduction systems to give a description of the deformation theory of arbitrary path algebras of quivers with relations and showed how this approach can be used to produce explicit combinatorial quantizations of Poisson structures by writing $\Bbbk [x_1, \dotsc, x_d] = \Bbbk \langle x_1, \dotsc, x_d \rangle / (x_j x_i - x_i x_j)_{1 \leq i < j \leq d}$ and systematically deforming the ideal of relations. In the present article we ``globalize'' this approach to describe the deformation theory of $\Qcoh (X)$ for any separated Noetherian scheme $X$.

The basic theoretical idea of \cite{barmeierwang} is to replace the bar resolution, giving rise to the Hochschild cochain complex, by a smaller, combinatorial resolution available for path algebras of (finite) quivers with relations. This resolution was constructed by M.~J.~Bardzell \cite{bardzell1,bardzell2} for monomial relations and by S.~Chouhy and A.~Solotar \cite{chouhysolotar} in the general case by using the combinatorics of reduction systems. We show in \S\ref{subsection:hypersurfaces} that for any affine hypersurface $A = \Bbbk [x_1, \dotsc, x_d] / (f)$ this resolution is isomorphic to the classical BACH resolution \cite{guccioneguccioneredondovillamayor}.

The deformation theory of path algebras of quivers with relations can then be studied via a natural L$_\infty$ algebra structure on the corresponding cochain complex of this smaller resolution, obtained by homotopy transfer from the DG Lie algebra structure on the (shifted) Hochschild cochain complex. We denote the resulting L$_\infty$ algebra by $\mathfrak p (Q, R)$, where $Q$ is the quiver and $R$ a reduction system for the ideal of relations. This point of view gives a systematic method of deforming the ideal of relations and allowing one to construct explicit formal deformations of such path algebras, as well as give sufficient criteria for the existence of an algebraization. 

In the present article we show how to use this point of view to give an explicit description of the deformation theory of the diagram algebra $\mathcal O_X \vert_{\mathfrak U}!$, where $X$ is any separated Noetherian scheme and $\mathcal O_X \vert_{\mathfrak U}$ is the restriction of the structure sheaf of $X$ to any affine open cover $\mathfrak U$ of $X$ which is closed under intersections. The approach via reduction systems allows one to deform the diagram of algebras by deforming the relations appearing in the diagram. By the above-mentioned results of \cite{lowenvandenbergh1,lowenvandenbergh2} this gives an explicit description of the deformation theory of $\Mod (\mathcal O_X)$, $\Qcoh (X)$ and $\coh (X)$. Our first main result can be phrased as follows.

\begin{theorem}[(Theorem \ref{theorem:diagramalgebrapqr})]
\label{theorem:introduction1}
Let $X$ be any separated Noetherian scheme and let $\mathfrak U$ be any affine open cover of $X$ closed under intersections. Then there is an explicit L$_\infty$ algebra $\mathfrak p (Q^{\mathrm{diag}}, R^{\mathrm{diag}})$ controlling the deformation theory of the diagram algebra $\mathcal O_X \vert_{\mathfrak U}!$ and of the Abelian categories $\Mod (\mathcal O_X)$, $\Qcoh (X)$ and $\coh (X)$.
\end{theorem}

This theorem relies on the construction of a quiver $Q^{\mathrm{diag}}$, whose vertices correspond to open sets in the cover $\mathfrak U$, and a finite reduction system $R^{\mathrm{diag}}$ which together encode the diagram algebra $\mathcal O_X \vert_{\mathfrak U}!$. Here, the Noetherian condition on the scheme $X$ ensures that the individual algebras $\mathcal O_X (U)$ for each $U \in \mathfrak U$ are finitely generated, so that the quiver $Q^{\mathrm{diag}}$ is finite. The reduction system $R^{\mathrm{diag}}$ is obtained by ``gluing'' finite reduction systems arising from Gröbner bases of the commutative algebras $\mathcal O_X (U)$ for $U \in \mathfrak U$. Deformations of the diagram algebra can then easily be viewed as deformations of the diagram $\mathcal O_X \vert_{\mathfrak U}$ and are thus very close to the geometry in the sense that the Abelian category of quasi-coherent modules over a deformation of $\mathcal O_X \vert_{\mathfrak U}$ can be viewed as a direct analogue of $\Qcoh (X)$.

The use of reduction systems and Gröbner bases gives a combinatorial description of the deformation theory of $\Qcoh (X)$ which is surprisingly workable. However, the number of generators and relations for the diagram algebra increase with the dimension and with the number of affine charts needed to cover $X$, so the following alternative description can also be useful.

In case $X$ admits a tilting bundle $\mathcal E$, the deformation theory of $\Mod (\mathcal O_X)$, $\Qcoh (X)$ and $\coh (X)$ can also be described more economically via the deformation theory of $\End \mathcal E$, which can also be naturally written as the path algebra of a quiver with relations, where now each vertex in the quiver corresponds to a direct summand of $\mathcal E$. The algebra $\End \mathcal E$ is always much ``smaller'' than the diagram algebra --- for example, it is finite dimensional if $X$ is projective --- but since one has to pass through a {\it derived} equivalence $\D (X) \simeq \D (\End \mathcal E)$, this point of view is not quite as close to the geometry. Our second main result is the following.

\begin{theorem}[(Theorem \ref{theorem:tiltingbundlepqr})]
\label{theorem:introduction2}
Let $X$ be any separated Noetherian scheme and let $\mathcal E$ be a tilting bundle on $X$. Write $\End \mathcal E = \Bbbk Q^{\mathrm{tilt}} / I$ and let $R^{\mathrm{tilt}}$ be any reduction system satisfying \textup{($\diamond$)} for $I$. Then $\mathfrak p (Q^{\mathrm{tilt}}, R^{\mathrm{tilt}})$ controls the deformation theory of the Abelian categories $\Mod (\mathcal O_X)$, $\Qcoh (X)$ and $\coh (X)$.
\end{theorem}

Theorems \ref{theorem:introduction1} and \ref{theorem:introduction2} should generally be understood in the context of {\it formal} deformation theory. In \cite[\S 9]{barmeierwang} we gave criteria for the existence of {\it algebraizations} of formal deformations, by using the notion of admissible orders $\prec$ on the set of paths of the quiver. In particular, this allows one to construct algebraic varieties $V_\prec$ of ``actual'' deformations obtained by evaluating the deformation parameters to a constant. In view of the Gabriel--Rosenberg reconstruction theorem \cite{gabriel,rosenberg,brandenburg}, the resulting deformations of the Abelian categories or algebras can be viewed as noncommutative schemes which are ``close to'' the original scheme and can be compared to it on equal footing.

We show in examples, namely for the smooth quasi-projective surfaces $Z_k = \Tot \mathcal O_{\mathbb P^1} (-k)$, how to use this idea to obtain explicit families of actual deformations for the Abelian categories $\Qcoh (Z_k)$, which also give rise to explicit families of generally noncommutative deformations of the $\frac1k (1,1)$ surface singularities.

\subsection{Structure of the article}

In \S\ref{section:reduction} we briefly recall the main results of \cite{barmeierwang} which allow us to study deformations of path algebras of quivers with relations via the combinatorics of reduction systems. In \S\ref{section:applications} we show how to apply these techniques to study the deformation theory of $\Qcoh (X)$ and prove Theorems \ref{theorem:introduction1} and \ref{theorem:introduction2}. In \S\ref{section:totalspaces} we illustrate the deformation--obstruction theory for the category of (quasi)coherent sheaves for a particular class of smooth quasi-projective surfaces, namely $Z_k = \Tot \mathcal O_{\mathbb P^1} (-k)$. In \S\ref{section:singularities} we show how to apply the deformation theory of path algebras of quivers with relations to study noncommutative deformations of singularities, which we illustrate for the cyclic surface singularities of type $\frac{1}k (1,1)$ by applying the results for the deformation theory of $\Qcoh (Z_k)$ obtained in \S\ref{section:totalspaces}.

\section{Deformations of path algebras of quivers with relations}
\label{section:reduction}

Let $Q$ be a finite quiver, let $\Bbbk Q$ denote its path algebra and let $I \subset \Bbbk Q$ be any two-sided ideal of relations. The quotient algebra $A = \Bbbk Q / I$ is a {\it path algebra of a quiver with relations} and any finitely generated algebra is of this form. In \cite{barmeierwang} we showed how to obtain a complete description of the deformation theory of $A$ for any finite quiver $Q$ and any ideal of relations $I$ via the combinatorics of reduction systems. We will give a brief summary of the construction here and refer to \cite{barmeierwang} for details.

\subsection{Deformations via the Hochschild complex}
\label{subsection:hochschild}

Classically, the formal deformation theory of an associative algebra $A$ is controlled by the Hochschild cochain complex $\mathrm C^\hdot (A, A) = (\Hom_{\Bbbk} (A^{\otimes_{\Bbbk} \hdot}, A), d)$ equipped with the Gerstenhaber bracket $[\blank {,} \blank]$.

We have a DG Lie algebra $\mathfrak h (A) = (\mathrm C^{\hdot + 1} (A, A), d, [\blank {,} \blank])$ whose Maurer--Cartan elements are precisely associative multiplications $A \otimes_{\Bbbk} A \tikzto A$. This DG Lie algebra naturally controls formal deformations of $A$.

Note that $\Hom_{\Bbbk} (A^{\otimes_{\Bbbk} n}, A) \simeq \Hom_{A^\e} (\Bar_n (A), A)$, where $\Bar_\ldot (A)$ is the bar resolution
\[
\dotsb \tikzto A \otimes_{\Bbbk} A^{\otimes_\Bbbk n} \otimes_{\Bbbk} A \tikzto \dotsb \tikzto A \otimes_{\Bbbk} A \otimes_{\Bbbk} A \tikzto A \otimes_{\Bbbk} A 
\]
which is a projective resolution of $A$ viewed as a bimodule over itself.

\subsection{Deformations via reduction systems}

If $A \simeq \Bbbk Q / I$ for some finite quiver $Q$ and some two-sided ideal $I \subset \Bbbk Q$, S.\ Chouhy and A.\ Solotar constructed a smaller $A$-bimodule resolution of $A$ via the combinatorics of reduction systems \cite{chouhysolotar}. This projective resolution can be used instead of the bar resolution to describe the full deformation theory of $A$. We briefly recall the basic notions and refer to \cite{barmeierwang} for more details.

\subsubsection{Reduction systems}

The following notion of a reduction system was introduced by G.~M.\ Bergman \cite{bergman} in the statement and proof of his Diamond Lemma.

\begin{definition}[{(Bergman \cite[\S 1]{bergman})}]
A {\it reduction system} $R$ for $\Bbbk Q$ is a set of pairs
\[
R = \{ (s, \varphi_s) \mid s \in S \text{ and } \varphi_s \in \Bbbk Q \}
\]
where
\begin{itemize}
\item $S$ is a subset of $Q_{\geq 2}$ such that $s$ is not a subpath of $s'$ when $s \neq s' \in S$
\item for all $s \in S$, $s$ and $\varphi_s$ are parallel
\item for each $(s, \varphi_s) \in R$, $\varphi_s$ is irreducible, i.e.\ it is a linear combination of irreducible paths.
\end{itemize}
Here a path is {\it irreducible} if it does not contain elements in $S$ as a subpath and we denote by $\Irr_S = \Irr_S (Q) = Q_\ldot \setminus Q_\ldot S \hair Q_\ldot$ the set of all irreducible paths.

Given a two-sided ideal $I$ of $\Bbbk Q$, we say that a reduction system $R$ {\it satisfies the condition \textup{($\diamond$)} for $I$} if
\begin{enumerate}
\item $I$ is equal to the two-sided ideal generated by the set $\{ s - \varphi_s \}_{(s, \varphi_s) \in R}$
\item every path is reduction-finite and reduction-unique, i.e.\ every path can be ``reduced'' by repeatedly replacing subpaths $s \in S$ by $\phi_s$ to a unique element in $\Bbbk \Irr_S$.
\end{enumerate}
\end{definition}

As a consequence of the Diamond Lemma \cite[Thm.~1.2]{bergman} it follows that if $R$ satisfies $(\diamond)$ for an ideal $I \subset \Bbbk Q$, then the set $\Irr_S$ of irreducible paths forms a $\Bbbk$-basis of the quotient algebra $A = \Bbbk Q / I$. We write
\[
\pi \colon \Bbbk Q \tikzto A = \Bbbk Q / I
\]
for the projection.

In particular, the deformation theory of $A = \Bbbk Q / I$ via a reduction system $R$ described below in \S\ref{subsubsection:linfinity} can be understood as deforming the algebra in the basis $\Irr_S$ determined by $R$.

\begin{remark}
\label{remark:f}
A reduction system $R = \{ (s, \varphi_s) \}$ is uniquely determined by the set $S \subset Q_{\geq 2}$ together with $\varphi \in \Hom (\Bbbk S, \Bbbk \Irr_S)$ so that $\phi (s) = \phi_s$ for each $s \in S$. (Here $\Hom$ denotes the set of $\Bbbk Q_0$-bimodule homomorphisms.)

Note that the Diamond Lemma implies in particular that we may identify $A \simeq \Bbbk \Irr_S$ and thus $\Hom (\Bbbk S, \Bbbk \Irr_S) \simeq \Hom (\Bbbk S, A)$.
\end{remark}

\begin{remark}
\label{remark:finite}
Reduction systems ``always exist''. More precisely, we have the following two facts:
\begin{enumerate}
\item For any finite quiver $Q$ and any two-sided ideal $I \subset \Bbbk Q$, there exists a reduction system $R$ satisfying $(\diamond)$ for $I$ (see \cite[Prop.~2.7]{chouhysolotar}). However, if $\Bbbk Q / I$ is noncommutative, it is in general undecidable whether there exists a {\it finite} reduction system, even if $I$ is finitely generated.
\item Any Gröbner basis for $I$ gives rise to a reduction system $R = \{ (s, \phi_s) \}$ (see e.g.\ \cite[\S 3.3.1]{barmeierwang}). In particular, if $\Bbbk Q / I$ is commutative, then $I$ always admits a finite Gröbner basis \cite{eisenbudpeevasturmfels}, which gives rise to a finite reduction system.
\end{enumerate}
\end{remark}

\subsubsection{Higher ambiguities}

We now recall the definition of $n$-ambiguities which will be used to construct a projective $A$-bimodule resolution of $A$.

\begin{definition}
Let $p \in Q_{\geq 0}$ be a path. If $p = qr$ for some paths $q, r$ we call $q$ a {\it proper left subpath} of $p$ if $p \neq q$. 

Now let $n \geq 0$. A path $p \in Q_\ldot$ is an {\it $n$-ambiguity} if there exist $u_0 \in Q_1$ and irreducible paths $u_1, \dotsc, u_{n+1}$ such that
\begin{enumerate}
\item $p = u_0 \dotsb u_{n+1}$
\item for all $i$, $u_i u_{i+1}$ is reducible, and $u_i d$ is irreducible for any proper left subpath $d$ of $u_{i+1}$.
\end{enumerate}
\end{definition}

Now let $S_0 = Q_0$, $S_1 = Q_1$, $S_2 = S$ and let $S_{n+2}$ for $n \geq 1$ denote the set of $n$-ambiguities. Generalizing Bardzell's resolution for monomial algebras \cite{bardzell1,bardzell2}, Chouhy--Solotar \cite{chouhysolotar} constructed a smaller $A$-bimodule resolution $P_\ldot$ of $A = \Bbbk Q/I$ associated to any reduction system $R = \{(s, \phi_s) \mid  s \in S\}$ satisfying ($\diamond$) for $I$
\begin{equation}
\label{resolution}
\dotsb \toarglong{\partial_{n+1}} P_n \toarglong{\partial_n} P_{n-1} \toarglong{\partial_{n-1}} \dotsb \toarglong{\partial_2} P_1 \toarglong{\partial_1} P_0 
\end{equation}
where $P_n = A \otimes \Bbbk S_n \otimes A$ and the augmentation map $\partial_0 \colon A \otimes A \tikzto A$ is given by the multiplication of $A$. A recursive formula for the differential is given in \cite[Thm.~4.2]{barmeierwang}.

\subsubsection{An explicit L$_\infty$ algebra}
\label{subsubsection:linfinity}

Using this smaller resolution, we obtained the following result. 

\begin{theorem}[{\cite[\S 7]{barmeierwang}}]
\label{theorem:BW1}
Let $A = \Bbbk Q / I$. Then there exists an L$_\infty$ algebra $\mathfrak p (Q, R)$
with underlying cochain complex $P^{\hdot+1} = \Hom_{A^\e} (P_{\ldot+1}, A)$ and an L$_\infty$ quasi-isomorphism between $\mathfrak p(Q, R)$ and the DG Lie algebra $\mathfrak h (A) = (\Hom_{A^\e} (\Bar_{\ldot+1} (A), A), d, [\blank {,} \blank])$.

As a consequence, the L$_\infty$ algebra $\mathfrak p (Q, R)$ controls the deformation theory of $A$.
\end{theorem}

For any complete local Noetherian $\Bbbk$-algebra $(B, \mathfrak m)$ --- the base of the deformation, e.g.\ $B = \Bbbk \llrr{t}$ and $\mathfrak m = (t)$ --- and any element $\widetilde \phi \in P^2 \hatotimes \mathfrak m$, we may associate a certain combinatorially defined operation $\star_{\phi + \widetilde \phi} \colon A \otimes A \tikzto A \hatotimes B$, which we call {\it combinatorial star product} (see \cite[Def.~7.18]{barmeierwang}). This operation can be described by performing rightmost reductions with respect to the new formal reduction system
\[
\widehat R_{\phi + \widetilde \phi} = \{ (s, \phi_s + \widetilde \phi_s) \mid s \in S\}
\]
where $\widetilde \phi_s \in \Bbbk \Irr_S \hatotimes \mathfrak m $ is the image of $\widetilde \phi(s) $ under the natural isomorphism $\Bbbk \Irr_S \simeq A$.

This can be understood as follows. Given any two irreducible paths $u, v \in \Irr_S$, the concatenation $u v$ is not necessarily irreducible. The multiplication in $A$ is given by repeatedly replacing any subpath $s$ of $uv$ with $s \in S$ by $\phi (s)$. Since the reduction system $R$ is assumed to be reduction-unique, this process finishes after finitely many steps and does not depend on the choice of subpath at each step. The operation $\star_{\phi + \widetilde \phi}$ is defined similarly with the only difference being that one starts with $u v$ and repeatedly replaces the rightmost subpath $s \in S$ by $(\phi + \widetilde \phi) (s)$. (See \cite[\S10]{barmeierwang} and Remark \ref{remark:graphical} for a graphical description of $\star_{\phi+\widetilde\phi}$ for the polynomial algebra $A = \Bbbk [x_1, \dotsc, x_d]$.)

This combinatorial operation $\star_{\phi + \widetilde \phi}$ can be used to give the following necessary and sufficient condition for $\widetilde \phi$ to be a Maurer--Cartan element.

\begin{theorem}[{\cite[Thm.~7.37]{barmeierwang}}]
\label{theorem:higher-brackets}
Let $\widetilde \varphi \in \Hom (\Bbbk S, A) \hatotimes \mathfrak m$ and write $\star = \star_{\phi + \widetilde \phi}$. Then $\widetilde \varphi$ satisfies the Maurer--Cartan equation of $\mathfrak p (Q, R) \hatotimes \mathfrak m$ if and only if for any $u v w \in S_3$ with $u v, v w \in S$, we have
\begin{align}
\label{mc}
\pi (u) \star (\pi (v) \star \pi (w)) = (\pi (u) \star \pi (v)) \star \pi (w).
\end{align}
\end{theorem}

This theorem shows that to check whether $\widetilde \varphi$ is a Maurer--Cartan element, it suffices to check whether $\star = \star_{\phi + \widetilde \phi}$ is associative on elements in $S_3$. Moreover, when $\widetilde \phi$ is a Maurer--Cartan element of $\mathfrak p (Q, R) \hatotimes \mathfrak m$, then $\star$ gives an explicit formula for the corresponding formal deformation of $A$. It follows from \cite[Cor.\ 7.31]{barmeierwang} that up to gauge equivalence any formal deformation of $A$ over $(B, \mathfrak m)$ is of the form $(A \hatotimes B, \star)$ for some Maurer--Cartan element $\widetilde \phi$.
 
\begin{remark}
\label{remark:computinghh2}
For $B = \Bbbk [t]/(t^2)$ Theorem \ref{theorem:higher-brackets} provides a combinatorial way to compute $\HH^2(A)$. Let $\widetilde \phi \in \Hom (\Bbbk S, A)$. Then $\widetilde \phi$ is a $2$-cocycle if and only if for any $u v w \in S_3$ with $u v, v w \in S$ we have
\[
(\pi (u) \star \pi(v)) \star \pi(w) = \pi (u) \star (\pi (v) \star \pi(w)) \; \mathrm{mod} \ t^2
\]
where $\star = \star_{\phi + \widetilde \phi t}$. See \cite[\S7.A]{barmeierwang} for more details. 
\end{remark}

\subsubsection{Algebraization and actual deformations}

Let $\widehat A = (A \llrr{t}, \star)$ be a formal one-para\-meter deformation of an associative algebra $A$. An {\it algebraization} of $\widehat A$ is a $C$-algebra $\widetilde A = (A \otimes C, \star')$, where $C$ is the coordinate ring of an affine curve, together with a smooth closed point of $\Spec C$ corresponding to a maximal ideal $\mathfrak m$ of $C$ such that $\widehat A$ is isomorphic to the $(A \otimes \mathfrak m)$-adic completion of $\widetilde A$ as formal deformation. 

When $A$ is infinite-dimensional over $\Bbbk$, an algebraization does not always exist. However, when it does, one can consider the formal deformation parameter $t$ as an actual parameter. For example, if there exists an algebraization with $C = \Bbbk [t]$ the coordinate ring of an affine line, the formal parameter $t$ can be evaluated to any constant $\lambda \in \Bbbk$, giving an ``actual'' deformation $\widetilde A_\lambda = \widetilde A / (t - \lambda)$ having the same $\Bbbk$-basis as $A$, but usually different geometric or representation-theoretic properties.

In \cite[\S 9]{barmeierwang} we use the notion of admissible orders on the set of paths for the quiver $Q$ to give criteria for the existence of algebraizations of formal deformations. More precisely, let $A = \Bbbk Q / I$ and let $R = \{ (s, \phi_s) \mid s \in S \}$ be a reduction system satisfying ($\diamond$) for $I$ obtained from a noncommutative Gröbner basis with respect to some admissible order $\prec$ on $Q_\ldot$. 

Now consider the subspace $\Hom (\Bbbk S, A)_\prec \subset \Hom (\Bbbk S, A)$ consisting of elements $\widetilde \varphi$ which satisfy the following degree condition:
\begin{flalign}
\label{degreecondition}
\tag{$\prec$}
&& \widetilde \varphi(s)  \in \Bbbk Q_{\prec s} && \mathllap{\text{for any $s \in S$}}
\end{flalign}
where $\Bbbk Q_{\prec s}$ is the $\Bbbk$-linear span of all paths which are ``smaller'' than $s$ with respect to the order $\prec$. Here we use the identification $\Hom(\Bbbk S, A) \simeq \Hom(\Bbbk S, \Bbbk \Irr_S)$ and thus view $\widetilde \phi(s) \in \Bbbk \Irr_S \cap \Bbbk Q_{\prec s}$ in \eqref{degreecondition}. Note that we have $\phi \in \Hom (\Bbbk S, A)_\prec$ since by definition $\phi(s)=\phi_s \prec s$ for any $s \in S$.

\begin{proposition}[{\cite[Prop.~9.17]{barmeierwang}}]
\label{proposition:Groebner}
Let $\widetilde \varphi \in \Hom (\Bbbk S, A)_\prec \otimes (t)$.
\begin{enumerate}
\item \label{lemgr1} For any $a, b \in A$ we have that $a \star_{\phi + \widetilde \varphi} b$ is a finite sum.
\item \label{lemgr2} If $\widetilde \phi$ satisfies the Maurer--Cartan equation of $\mathfrak p (Q, R) \hatotimes (t)$, then the formal deformation $(A \llrr{t}, \star_{\phi + \widetilde \varphi})$ admits $(A [t], \star_{\phi + \widetilde \phi})$ as an algebraization.
\end{enumerate}
\end{proposition}

In particular, in \cite{barmeierwang} it was shown that when $R$ is a finite reduction system, the set $V_{\prec}$ of Maurer--Cartan elements in $\Hom(\Bbbk S, A)_\prec$ of $\mathfrak p(Q, R)$ admits the structure of an affine algebraic variety of algebras previously studied by E.\ Green, L.\ Hille and S.\ Schroll \cite{greenhilleschroll}. This variety has the following properties.

\begin{theorem}[{\cite[Thm.~2.10]{barmeierwang}}]
\label{theorem:variety}
Let $\dim \Hom (\Bbbk S, A)_\prec = N < \infty$.
\begin{enumerate}
\item \label{variety2} There is a natural groupoid action $G_\prec \tikzrightrightarrows V_\prec$ corresponding to equivalence of reduction systems, such that two Maurer--Cartan elements in $V_\prec$ are equivalent if and only if they lie in the same orbit of $G_\prec \tikzrightrightarrows V_\prec$.
\item The Zariski tangent space of $V_\prec$ at the point $\widetilde \phi \in V_\prec$ is isomorphic to the space of $2$-cocycles in $\Hom (\Bbbk S, A_{\phi + \widetilde \phi})_\prec$. In particular, we have a \emph{Kodaira--Spencer map} $\mathrm{KS} \colon \mathrm T_{\widetilde \phi} V_\prec \tikzto \HH^2 (A_{\phi + \widetilde \phi}, A_{\phi + \widetilde \phi})$.
\item The Zariski tangent space to the orbit of the groupoid $G_\prec \tikzrightrightarrows V_\prec$ at $\widetilde \phi$ is contained in the subspace of the $2$-co\-bound\-aries lying in $\Hom (\Bbbk S, A_{\phi + \widetilde \phi})_\prec$.
\end{enumerate}
\end{theorem}

\subsection{Hypersurfaces}
\label{subsection:hypersurfaces}

Affine hypersurfaces, i.e.\ varieties whose coordinate ring is of the form $A = \Bbbk [x_1, \dotsc, x_d] / (f)$ with $\deg f = n \geq 2$, are an interesting class of algebraic varieties. A projective $A$-bimodule resolution of affine hypersurfaces was given by the Buenos Aires Cyclic Homology Group (BACH) in \cite{guccioneguccioneredondovillamayor}. In the following, we show that this resolution can be viewed as a special case of the Bardzell--Chouhy--Solotar resolution $P_\ldot$ \eqref{resolution}.

Let $Q$ be the quiver with one vertex and $d$ loops $x_1, \dotsc, x_d$. Write $A = \Bbbk Q / I$, where the ideal $I$ is generated by the set
\[
\{f\} \cup \{x_j x_i - x_i x_j\}_{1 \leq i < j \leq d}.
\]
After a linear change of coordinates we may assume that $f$ has leading term $x_d^n$ with respect to the lexicographic ordering $x_1\prec x_2 \prec \dotsb \prec x_d$. We may thus write 
\[
f = x_d^n + \sum_{\substack{0 \leq i_1, \dotsc, i_d \leq n \\ i_1 + \dotsb + i_d \leq n}} \lambda_{i_1, \dotsc, i_d} x_1^{i_1} x_2^{i_2}\dotsb x_d^{i_d}
\] 
with $\lambda_{i_1, \dotsc, i_d} \in \Bbbk$ and $\lambda_{0, \dotsc, 0, n} = 0$. Then we have a reduction system
\[
R = \bigl\{ \bigl( x_d^n, - \textstyle\sum \lambda_{i_1, \dotsc, i_d} x_1^{i_1} x_2^{i_2}\dotsb x_d^{i_d} \bigr) \bigr\} \cup \bigl\{ (x_j x_i, x_i x_j) \bigr\}_{1 \leq i < j \leq d}
\]
which satisfies the condition ($\diamond$) for $I$.
We have $S = \{ x_d^n\} \cup \{x_{i_2}x_{i_1}\}_{1 \leq i_1< i_2\leq d}$.

Set $S_0 = Q_0 = \{\bullet \}$ and $S_1 = Q_1 = \{x_i\}_{1\leq i \leq d}$.  For $m \geq 0$ the set of $m$-ambiguities 
\[
S_{m+2} = \{x_d^{nj} x_{i_k}  \dotsb x_{i_2} x_{i_1} \mid  j, k \geq 0,  \  1 \leq i_1 < \dotsb < i_k \leq d\quad  \text{and}\quad  2j + k = m+2 \}. 
\]
In particular, we have $S_2 = S$ and \eqref{resolution} gives an explicit resolution $P_\ldot$ of $A$.

\begin{proposition}
The resolution $P_\ldot$ is isomorphic to the BACH resolution $C_\ldot$.
\end{proposition}

\begin{proof}
Let $e_{i_1}\dotsb e_{i_k} t^{(j)}$ be a basis element of the BACH resolution (see \cite[\S 2.3]{guccioneguccioneredondovillamayor}). The map $P_\ldot \tikzto C_\ldot$ given by
\[ 
x_d^{jn} x_{i_k} x_{i_{k-1}} \dotsb x_{i_1} \tikzmapsto e_{i_1}\dotsb e_{i_k} t^{(j)}
\]
is an isomorphism of complexes.
\end{proof}

 \begin{remark}
Under the above identification, by Theorem \ref{theorem:BW1} we obtain an L$_\infty$ algebra structure on the BACH complex $C^{\hdot+1} = \Hom_{A^\e}(C_{\ldot+1}, A)$. 
\end{remark}

The underlying complex of the L$_\infty$ algebra $\mathfrak p (Q, R)$ is $\Hom(\Bbbk S_{\ldot+1}, A)$. Note that we have the following natural identification 
\[
\Hom(\Bbbk S_{i+1}, A) \simeq  \bigoplus_{y \in S_{i+1} } A e_{y}
\]
where $e_{y}$ denotes the dual basis of $y \in S_{i+1}$. Under this identification, the differential of $\mathfrak p(Q, R)$ is given by 
\[
\partial (a  e_{x_d^{nj} x_{i_k}  \dotsb x_{i_2} x_{i_1}}) = \sum_{l=1}^k (-1)^{l-1}a\frac{\partial f}{\partial x_{i_l}} e_{x_d^{n (j+1)} x_{i_k} \dotsb \widehat x_{i_l} \dotsb x_{i_1}}
\]
where $a \in A.$ In particular, we have $\partial (a  e_{x_d^{nj}}) =0$.
Thus we have the following result (also see \cite[Thm.~3.2.7]{guccioneguccioneredondovillamayor}).

\begin{lemma}
\label{lemma:hh2hypersurface}
$\HH^2 (A) \simeq A  / \bigl( \frac{\partial{f}}{\partial{x_1}}, \dotsc, \frac{\partial{f}}{\partial{x_d}} \bigr) \oplus \mathcal N$,
where 
\[
\mathcal N = \Bigl\{ \textstyle\sum\limits_{1\leq i<j\leq d} a_{ji} e_{x_jx_i} \Bigm| \sum\limits_{i = 1}^{d} (a_{ki} - a_{ik})  \frac{\partial{f}}{\partial{x_i}}  = 0 \ \textup{for each $1 \leq k \leq d$} \Bigr\}.
\]
Here we set $a_{ji} = 0$ if $j \leq i$.
\end{lemma}

Note that the first summand is isomorphic to the Harrison cohomology $\Har^2 (A)$, which corresponds to commutative deformations of $A$. If $A$ is regular, then $\Har^2 (A) = 0$ and the second summand is the space of bivector fields on $\Spec A$.

\subsubsection{Singular hypersurfaces}

We now give a first affine example by using the deformation--obstruction calculus for the L$_\infty$ algebra $\mathfrak p (Q, R)$.

\begin{example}
Consider $f = x_3^{n} - x_1x_2 \in \Bbbk [x_1, x_2, x_3]$ for $n \geq 2$. Let $A = \Bbbk Q/I = \Bbbk [x_1, x_2, x_3]/(f)$.  We have a reduction system satisfying the condition  ($\diamond$) for $I$
\begin{align*}
R &= \{ (x_3^n, x_1x_2) \} \cup \{(x_2x_1, x_1x_2), (x_3x_1, x_1x_3), (x_3x_2, x_2x_3)\} \\
\text{with}\qquad S &= \{x_3^n, x_2x_1, x_3x_1, x_3x_2\} \\
\text{and}\qquad S_3 &= \{ x_3^n x_1, x_3^n x_2,  x_3^{n+1}, x_3x_2x_1\}.
\end{align*}
Consider the element $\widetilde \phi \in \Hom(\Bbbk S, A) \otimes (t)$ given by 
\begin{flalign*}
&& \widetilde \phi(x_3^n)  &= 0                 & \widetilde \phi(x_3x_1) &= x_1t &&\\
&& \widetilde \phi(x_2x_1) &= (x_3+t)^n - x_3^n & \widetilde \phi(x_3x_2) &= -x_2 t. &&
\end{flalign*}
Write $\star = \star_{\phi + \widetilde \phi}$, where $\phi \in \Hom(\Bbbk S, A)$ is the element corresponding to $R$ (see Remark \ref{remark:f}). Let us verify \eqref{mc} for the elements in $S_3$.  Using the graphical description of $\star$ in \cite[\S10]{barmeierwang}, we have 
\[
x_3^{n-1} \star ( x_3 \star x_1) = (x_3^{n-1} \star x_3) \star x_1
\]
since both sides equal 
\[
x_1^2 x_2 + \sum_{i=1}^n {n \choose i} x_1 x_3^{n-i} t^i.
\]
Similarly, we may verify \eqref{mc} for the other elements $x_3^nx_2, x_3^{n+1}, x_3x_2x_1$ in $S_3$. It follows from Theorem \ref{theorem:higher-brackets} that 
$(A \llrr{t}, \star)$ is a formal deformation of $A$. We also note that $(A \llrr{t}, \star)$ is a deformation quantization of the (exact) Poisson bracket on $A$ determined by
\begin{flalign*}
&& [x_3, x_1] = - \frac{\partial f}{\partial x_2} && [x_3, x_2] = \frac{\partial f}{\partial x_1}  && [x_2, x_1] = \frac{\partial f}{\partial x_3}.&& 
\end{flalign*}
\end{example}

\begin{remark}\label{remark:hypersurfaces}
Let $n = 2$. Evaluating the deformation $(A\llrr{t}, \star)$ at $t = 1$ we obtain that $(A, \star)$ is isomorphic to $\widetilde A = \U (\mathfrak{sl}_2) / (C + \tfrac12)$, where $C = XY + YX + \frac12 H^2$ is the Casimir element of $\U (\mathfrak{sl}_2)$. (This can be seen by constructing an algebra isomorphism which sends $X$ to $-x_1$, $Y$ to $x_2$ and $H$ to $2x_3+1$.)

The category $\Mod (\widetilde A)$ is thus an actual deformation of $\Mod (A) \simeq \Qcoh (\Spec A)$ (cf.\ \eqref{eq}), so that $\Mod (\widetilde A)$ can be viewed as the category of quasi-coherent sheaves on a ``singular quantum hypersurface''.

By Proposition \ref{proposition:z2} \ref{Z2-4} the above deformation can be obtained from the deformed preprojective algebra of type $\widetilde{\mathrm A}_{1}$. Similarly, for general $n \geq 2$, we may recover the deformation $(A\llrr{t}, \star)$ from the deformed preprojective algebra of type $\widetilde{\mathrm A}_{n-1}$ studied in \cite{crawleyboeveyholland}.
\end{remark}

\section{Deformations of categories of coherent sheaves}
\label{section:applications}

In this section we show how deformations of the Abelian category $\Qcoh (X)$ can be described as deformations of a path algebra of a suitable quiver with relations. In this context, there are two main sources of such quivers with relations: one is defined from any affine open cover $\mathfrak U$ of $X$, and the other is defined from a tilting bundle on $X$, if available. The former works in greater generality, but when $X$ admits a tilting bundle, the latter can be computationally more convenient.

Let $X$ be a separated Noetherian scheme over an algebraically closed field $\Bbbk$ of characteristic $0$, so that $X$ admits a finite affine open cover $\mathfrak U$ which is closed under intersections.

The restriction $\mathcal O_X \vert_{\mathfrak U}$ of the structure sheaf to the cover $\mathfrak U$ can be viewed as a {\it diagram of algebras} (see Definition \ref{definition:diagramalgebra}) which is a contravariant functor (a presheaf) $\mathcal O_X \vert_{\mathfrak U} \colon \mathfrak U \tikzto \mathfrak{Alg}_{\Bbbk}$, where $\mathfrak U$ can be viewed as a finite subcategory of the category $\mathfrak{Open} (X)$ of open sets with morphisms given by inclusion. Deformations of diagrams of algebras were first studied by Gerstenhaber--Schack \cite{gerstenhaberschack} and higher structures on the Gerstenhaber--Schack complex were given in Dinh Van--Lowen \cite{dinhvanlowen} and also in Dinh Van--Hermans--Lowen \cite{dinhvanhermanslowen} using operads. (See also \cite{barmeierfregier} for a construction via higher derived brackets.)

The following result establishes an equivalence of different deformation problems.

\begin{theorem}[(Lowen--Van den Bergh \cite{lowenvandenbergh1,lowenvandenbergh2})]
\label{theorem:equivalence}
Let $(X, \mathcal O_X)$ be a quasi-compact and separated scheme over an algebraically closed field $\Bbbk$ of characteristic $0$.

There is an equivalence of formal deformations between deformations of
\begin{enumerate}
\item $\mathcal O_X \rvert_{\mathfrak U}$ as twisted presheaf
\item \label{eqcoh2} $\mathcal O_X \rvert_{\mathfrak U}!$ as associative algebra
\item $\Qcoh (X)$ as Abelian category
\item $\Mod (\mathcal O_X)$ as Abelian category.
\end{enumerate}
Moreover, if $X$ is Noetherian, then the above deformations are also equivalent to deformations of
\begin{enumerate}
\setcounter{enumi}{4}
\item $\coh (X)$ as Abelian category.
\end{enumerate}
\end{theorem}

In the remainder we phrase all statements for deformations of $\Qcoh (X)$ but since we work with Noetherian schemes, Theorem \ref{theorem:equivalence} shows that analogous statements hold for $\Mod (\mathcal O_X)$ and $\coh (X)$.

The different types of deformations in Theorem \ref{theorem:equivalence} are parametrized by what are essentially various versions of Hochschild cohomology with first-order deformations parametrized by
\begin{equation}
\label{cohomologies}
\HH^2 (X) \simeq \H^2_{\mathrm{GS}} (\mathcal O_X \vert_{\mathfrak U}) \simeq \HH^2 (\mathcal O_X \vert_{\mathfrak U}!) \simeq \H_{\mathrm{Ab}}^2 (\Qcoh (X))
\end{equation}
and obstructions in $\HH^3 (X) \simeq \dotsb \simeq \H_{\mathrm{Ab}}^3 (\Qcoh (X))$. Here $\HH^\hdot (A) = \Ext^\hdot_{A^\e} (A, A)$ is the usual Hochschild cohomology for associative algebras, and the Hochschild cohomology of a scheme may be defined analogously as
\begin{equation}
\label{hochschildscheme}
\HH^\hdot (X) := \Ext^\hdot_{\mathcal O_{X \times X}} (\delta_* \mathcal O_X, \delta_* \mathcal O_X)
\end{equation}
where $\delta_* \mathcal O_X$ is the pushforward of the structure sheaf along the diagonal map $\delta \colon X \tikzto X \times X$ \cite{gerstenhaberschack,kontsevich1,swan}. Also, $\H^\hdot_{\mathrm{GS}}$ is a cohomology theory for diagrams (or prestacks) of algebras (see \cite{dinhvanlowen,gerstenhaberschack}) and $\H^\hdot_{\mathrm{Ab}}$ a cohomology theory for Abelian categories (see \cite{lowenvandenbergh1,lowenvandenbergh2}).

\subsubsection{Geometry}

Theorem \ref{theorem:equivalence} shows that the deformation theory of the Abelian category $\Qcoh (X)$ admits several equivalent algebraic descriptions --- namely deformations of $\mathcal O_X \vert_{\mathfrak U}$ as a twisted presheaf, or deformations of the diagram algebra $\mathcal O_X \vert_{\mathfrak U}!$ as an associative algebra.

The isomorphisms (\ref{cohomologies}) showed that the cohomology groups parametrizing these deformations are isomorphic to the Hochschild cohomology $\HH^2 (X)$ of the scheme, furnishing the following geometric interpretation.

If $X$ is smooth, then the Hochschild--Kostant--Rosenberg theorem (see \cite{yekutieli1}) gives a decomposition
\begin{equation}
\label{hkr}
\HH^n (X) \simeq \bigoplus_{p+q=n} \H^p (\Lambda^q \mathcal T_X)
\end{equation}
where $\Lambda^q \mathcal T_X$ is the sheaf of sections in the $q$th exterior power of the tangent bundle. A similar geometric interpretation can be given for singular $X$ for which a decomposition of $\HH^\hdot (X)$ was given in \cite{buchweitzflenner1,buchweitzflenner2}.

First-order deformations of $\Qcoh (X)$ are thus parametrized by
\[
\HH^2 (X) \simeq \H^0 (\Lambda^2 \mathcal T_X) \oplus \H^1 (\mathcal T_X) \oplus \H^2 (\mathcal O_X)
\]
where
\begin{enumerate}
\item $\H^0 (\Lambda^2 \mathcal T_X)$ is the space of bivector fields, which for Poisson bivector fields parametrize a (noncommutative) algebraic quantization of $\mathcal O_X$
\item $\H^1 (\mathcal T_X)$ is well known to parametrize algebraic deformations of $X$ as a scheme (which over $\Bbbk = \mathbb C$ corresponds to classical deformations of the complex structure), and
\item $\H^2 (\mathcal O_X)$ parametrizes ``twists'', i.e.\ deformations of the (trivial) $\mathcal O_X^*$-gerbe structure of $\mathcal O_X$.
\end{enumerate}
Hence deformations of $\Qcoh (X)$ can be thought of as a combination of these three types of deformations.

Note that if $X$ is a curve, then $\HH^2 (X) \simeq \H^1 (\mathcal T_X)$ and $\HH^3 (X) = 0$ so that all deformations of $\Qcoh (X)$ are induced by classical deformations of the curve (cf.\ Example \ref{example:genus3}).

\subsection{Deformations via the diagram algebra}
\label{subsection:diagram}

First we shall consider \ref{eqcoh2} of Theorem \ref{theorem:equivalence} which concerns deformations of the so-called {\it diagram algebra} $\mathcal O_X \vert_{\mathfrak U}!$ which is defined as follows.

\begin{definition}
\label{definition:diagramalgebra}
A {\it diagram of $\Bbbk$-algebras} over a small category $\mathfrak U$ is a contravariant functor $\mathcal A \colon \mathfrak U \tikzto \mathfrak{Alg}_{\Bbbk}$ to the category of associative $\Bbbk$-algebras.

The {\it diagram algebra} of $\mathcal A$, denoted $\mathcal A!$, is given as $\Bbbk$-module by
\[
\mathcal A! = \prod_{U \in \mathfrak U} \, \bigoplus_{f \colon U \smallto V} \mathcal A (U) x_f
\]
where the sum is over all morphisms in $\mathfrak U$, and $x_f$ is simply a formal (bookkeeping) symbol. The multiplication of elements $a \in \mathcal A (U)$ and $b \in \mathcal A (V)$ is defined by
\begin{flalign*}
&& (a x_f) (b x_g) =
\begin{cases}
a \mathcal A (f) (b) x_{g \circ f} & \text{if $g \circ f$ is defined} \\
0 & \text{otherwise}
\end{cases} &&
\end{flalign*}
where $a \mathcal A (f) (b)$ is the product of $a$ and $\mathcal A (f) (b)$ in the algebra $\mathcal A (U)$
\[
\begin{tikzpicture}[baseline=-2.75pt,x=5.5em,y=2em]
\node (U) at (0,0) {$U$};
\node (V) at (1,0) {$V$};
\node (W) at (2,0) {$W$};
\node (AU) at (0,-1) {$\mathcal A (U)$};
\node (AV) at (1,-1) {$\mathcal A (V)$};
\node (AW) at (2,-1) {$\mathcal A (W)$\rlap{.}};
\path[->,line width=.4pt,font=\scriptsize]
(U) edge node[above=-.2ex] {$f$} (V)
(V) edge node[above=-.2ex] {$g$} (W)
(AV) edge node[above=-.2ex,pos=.45] {$\mathcal A (f)$} (AU)
(AW) edge node[above=-.2ex,pos=.45] {$\mathcal A (g)$} (AV)
;
\end{tikzpicture}
\]
\end{definition}

\begin{definition}
\label{definition:modules}
Let $\mathcal A$ be a diagram of algebras. An {\it $\mathcal{A}$-module} $\mathcal{M}$ is a contravariant functor $\mathcal M \colon \mathfrak U \tikzto \mathfrak{Vec}_{\Bbbk} $ to the category of $\Bbbk$-vector spaces such that  $\mathcal M(U)$ is an $\mathcal{A}(U)$-module for each $U \in \mathfrak U$ and  $\mathcal{M}(f) \colon \mathcal M(V)\tikzto \mathcal M(U)$ is a morphism of $\mathcal A(V)$-modules for each morphism $f \colon U \tikzto V$.

An $\mathcal A$-module $\mathcal M$ is {\it quasi-coherent} if the induced map $\mathcal A(U) \otimes_{\mathcal A(V)} \mathcal M(V) \tikzto \mathcal M(U)$ is an isomorphism for any $f \colon U \tikzto V$ (see Enochs--Estrada \cite[\S2]{enochsestrada}). We denote by $\Mod (\mathcal A)$ the Abelian category of all $\mathcal A$-modules and by $\Qcoh (\mathcal A)$ the full subcategory of quasi-coherent modules.
\end{definition}

This definition of modules over a {\it diagram} of algebras is compatible with the usual notion of modules in that we have an equivalence of Abelian categories $\Mod (\mathcal A) \simeq \Mod (\mathcal A!)$. We denote by $\Qcoh(\mathcal A!)$ the full subcategory of the images of $\Qcoh(\mathcal A)$. It follows from \cite[\S2]{enochsestrada} that for any separated Noetherian scheme $X$ and any finite affine open cover $\mathfrak{U}$ which is closed under intersections, we have
\[
\Qcoh(X) \simeq \Qcoh(\mathcal O_X \vert_{\mathfrak U}) \simeq \Qcoh(\mathcal O_X \vert_{\mathfrak U}!).
\]

Theorem \ref{theorem:equivalence} states that studying deformations of $\Qcoh (X)$ as Abelian category is equivalent to studying deformations of the diagram algebra $\mathcal O_X \vert_{\mathfrak U}!$. The construction summarized in the following proposition can be used to study deformations of $\Qcoh (X)$ for $X$ any separated Noetherian scheme via the approach outlined in \S\ref{section:reduction} for deformations of $\Bbbk Q / I$.

\begin{proposition}
\label{proposition:diagram}
Let $X$ be any separated Noetherian scheme and let $\mathfrak U$ be any finite affine open cover of $X$ which is closed under intersections.

The diagram algebra $\mathcal O_X \vert_{\mathfrak U}!$ is isomorphic to the path algebra of a finite quiver $Q^{\mathrm{diag}}$ on $\# \mathfrak U$ vertices modulo a finitely generated ideal $J$ of relations. Moreover, there exists a finite reduction system $R^{\mathrm{diag}}$ satisfying \textup{($\diamond$)} for $J$.
\end{proposition}

\begin{proof}
We give a general constructive proof, but for concrete examples see Example \ref{example:genus3} and \S\ref{subsection:zkdiagram}.

Let $U_1, \dotsc, U_n$ be an affine open cover of $X$ and let $\mathfrak U = \{ U_{i_1 \cdots i_m} \mid 1 \leq i_1 < \dotsb < i_m \leq n \}_{1 \leq m \leq n}$ where
\[
U_{i_1 \cdots i_m} = U_{i_1} \cap \dotsb \cap U_{i_m}
\]
i.e.\ $\mathfrak U$ is the closure of $\{ U_1, \dotsc, U_n \}$ under taking intersections.

The diagram $\mathcal O_X \vert_{\mathfrak U}$ can be viewed as the $1$-skeleton of an $n$-hypercube with one vertex (corresponding to $X$) and the incident edges removed. The cover $\mathfrak U$ has cardinality $2^n - 1$ and the diagram algebra $\mathcal O_X \vert_{\mathfrak U}!$ can then be written as the path algebra $\Bbbk Q^{\mathrm{diag}} / J$ of a quiver $Q^{\mathrm{diag}}$ with relations $J$ as follows.

Let $Q^{\mathrm{diag}}$ be the quiver on $2^n - 1$ vertices, each vertex corresponding to an open set $U_{i_1 \cdots i_m} \in \mathfrak U$ and let each vertex be labelled by $i_1 \cdots i_m$, say. To simplify the indexing, let us write $\mathbf i$ for some label $i_1 \cdots i_m$ and if $m < n$ and $j \in \{ 1, \dotsc, n \} \setminus \{ i_1, \dotsc, i_m \}$ we write $\mathbf i j$ for the label $i_1 \cdots j \cdots i_m$ obtained by adding $j$ to $\mathbf i$. We will now add several sets of arrows to the quiver.

If $\mathbf i = i_1 \cdots i_m$, there are $n-m$ inclusions $U_{\mathbf i j} \subset U_{\mathbf i}$, so for each such inclusion we add an arrow
\[
\mathbf i j \longleftarrowarg{f^{\mathbf i}_j} \mathbf i
\]
in the opposite direction since $\mathcal O_X \vert_{\mathfrak U}$ is contravariant. For example, for $n = 1, 2, 3$ this gives the following acyclic quivers
\begin{equation}
\label{acyclicpn}
\begin{tikzpicture}[baseline=-2.75pt,line width=.4pt,x=3.75em,y=1.75em]
\node[font=\scriptsize,inner sep=1.3pt] (L)  at (0,0) {$01$};
\node[font=\scriptsize,inner sep=1.3pt] (R1)  at (1,.5) {$0$};
\node[font=\scriptsize,inner sep=1.3pt] (R2)  at (1,-.5) {$1$};
\path[->]
(R1) edge (L)
(R2) edge (L)
;
\begin{scope}[shift={(8em,0)}]
\node[font=\scriptsize,inner sep=1.3pt] (L)  at (0,0) {$012$};
\node[font=\scriptsize,inner sep=1.3pt] (M1) at (1,1) {$01$};
\node[font=\scriptsize,inner sep=1.3pt] (M2) at (1,0) {$02$};
\node[font=\scriptsize,inner sep=1.3pt] (M3) at (1,-1) {$12$};
\node[font=\scriptsize,inner sep=1.3pt] (R1) at (2,1) {$0$};
\node[font=\scriptsize,inner sep=1.3pt] (R2) at (2,0) {$1$};
\node[font=\scriptsize,inner sep=1.3pt] (R3) at (2,-1) {$2$};
\path[->]
(M1) edge (L.25)
(M2) edge (L)
(M3) edge (L.335)
(R1) edge (M1)
(R1) edge (M2.10)
(R2) edge (M1)
(R2) edge (M3)
(R3) edge (M2.350)
(R3) edge (M3)
;
\end{scope}
\begin{scope}[shift={(20em,0)}]
\node[font=\scriptsize,inner sep=1.3pt] (0123)  at (0,0) {$0123$};
\node[font=\scriptsize,inner sep=1.3pt] (012) at (1,1.5) {$012$};
\node[font=\scriptsize,inner sep=1.3pt] (013) at (1,.5) {$013$};
\node[font=\scriptsize,inner sep=1.3pt] (023) at (1,-.5) {$023$};
\node[font=\scriptsize,inner sep=1.3pt] (123) at (1,-1.5) {$123$};
\node[font=\scriptsize,inner sep=1.3pt] (01) at (2,2.5) {$01$};
\node[font=\scriptsize,inner sep=1.3pt] (02) at (2,1.5) {$02$};
\node[font=\scriptsize,inner sep=1.3pt] (03) at (2,.5) {$03$};
\node[font=\scriptsize,inner sep=1.3pt] (12) at (2,-.5) {$12$};
\node[font=\scriptsize,inner sep=1.3pt] (13) at (2,-1.5) {$13$};
\node[font=\scriptsize,inner sep=1.3pt] (23) at (2,-2.5) {$23$};
\node[font=\scriptsize,inner sep=1.3pt] (0) at (3,1.5) {$0$};
\node[font=\scriptsize,inner sep=1.3pt] (1) at (3,.5) {$1$};
\node[font=\scriptsize,inner sep=1.3pt] (2) at (3,-.5) {$2$};
\node[font=\scriptsize,inner sep=1.3pt] (3) at (3,-1.5) {$3$};
\path[->]
(012) edge (0123.25)
(013) edge (0123.9)
(023) edge (0123.351)
(123) edge (0123.335)
(01) edge (012)
(01) edge (013.20)
(02) edge (012)
(02) edge (023.20)
(03) edge (013)
(03) edge (023.350)
(12) edge (012.340)
(12) edge (123)
(13) edge (013.340)
(13) edge (123)
(23) edge (023.340)
(23) edge (123)
(0) edge (01.10)
(0) edge (02)
(0) edge (03.10)
(1) edge (01.340)
(1.180) edge (12)
(1) edge (13.20)
(2) edge (02.340)
(2) edge (12)
(2) edge (23.20)
(3.140) edge (03.350)
(3) edge (13)
(3.220) edge (23.350)
;
\end{scope}
\end{tikzpicture}
\end{equation}
For each square
\begin{equation}
\label{square}
\begin{tikzpicture}[baseline=-2.75pt,line width=.4pt,x=-2em,y=1.75em]
\node[font=\footnotesize] (L)  at (0,0) {$\mathbf i$};
\node[font=\footnotesize] (M1) at (2,1) {$\mathbf i j$};
\node[font=\footnotesize] (M2) at (2,-1) {$\mathbf i k$};
\node[font=\footnotesize] (R)  at (4,0) {$\mathbf i j k$};
\path[->]
(L) edge node[font=\scriptsize,above=-.2ex] {$f^{\mathbf i}_j$} (M1)
(L) edge node[font=\scriptsize,below=-.2ex] {$f^{\mathbf i}_k$} (M2)
(M1) edge node[font=\scriptsize,above=-.2ex] {$f^{\mathbf i j}_k$} (R.20)
(M2) edge node[font=\scriptsize,below=-.2ex] {$f^{\mathbf i k}_j$} (R.345)
;
\end{tikzpicture}
\end{equation}
we add the relation
\begin{equation}
\label{relationsquare}
f^{\mathbf i}_k f^{\mathbf i k}_j - f^{\mathbf i}_j f^{\mathbf i j}_k
\end{equation}
since the diagram $\mathcal O_X \vert_{\mathfrak U}$ is commutative. Let $J_2$ denote the ideal generated by the relations \eqref{relationsquare} for all squares of the form \eqref{square}. Let
\[
R_2 = \bigcup_{1 \leq \lvert \mathbf i \rvert \leq n - 2} \{ (f^{\mathbf i}_k f^{\mathbf i k}_j, f^{\mathbf i}_j f^{\mathbf i j}_k) \}_{j < k}
\]
where the length $\lvert \mathbf i \rvert$ of the label is assumed to be at most $n - 2$ so that adding two distinct indices is possible. Then $R_2$ is a reduction system which satisfies $(\diamond)$ for the ideal $J_2$, since reductions correspond to reordering the lower indices to be strictly increasing. (Any overlap for $R_2$ is of the form $f^{\mathbf i}_l f^{\mathbf i l}_k f^{\mathbf i k l}_j$ for some $j < k < l$, which uniquely resolves to $f^{\mathbf i}_j f^{\mathbf i j}_k f^{\mathbf i j k}_l$.)

Now at each vertex $\mathbf i$ we have a finitely generated commutative algebra
\[
\mathcal O_X (U_{\mathbf i}) \simeq \Bbbk [x^{\mathbf i}_1, \dotsc, x^{\mathbf i}_{N_{\mathbf i}}] / (F^{\mathbf i}_1, \dotsc, F^{\mathbf i}_{M_{\mathbf i}}).
\]
This algebra can be written as the path algebra of a quiver of a single vertex with $N_{\mathbf i}$ loops $x^{\mathbf i}_1, \dotsc, x^{\mathbf i}_{N_{\mathbf i}}$ modulo the ideal $J^{\mathbf i}$ generated by commutativity relations $x^{\mathbf i}_j x^{\mathbf i}_i = x^{\mathbf i}_i x^{\mathbf i}_j$ and the relations $F^{\mathbf i}_1, \dotsc, F^{\mathbf i}_{M_{\mathbf i}}$. For each vertex $\mathbf i$ fix any finite reduction system $R^{\mathbf i}$ satisfying $(\diamond)$ for $J^{\mathbf i}$. (Note that such a finite reduction system exists by Remark \ref{remark:finite} and can be explicitly computed from a commutative Gröbner basis.) At each vertex $\mathbf i$ we thus add $N_{\mathbf i}$ loops to the quiver $Q^{\mathrm{diag}}$ and set
\[
R_0 = \bigcup_{1 \leq \lvert \mathbf i \rvert \leq n} R^{\mathbf i}.
\]

We need to add one last set of relations involving the loops at two vertices which are connected by an arrow of the underlying acyclic quiver. For each $f^{\mathbf i}_j$ we have
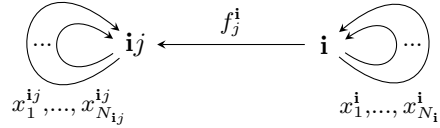
\[
\begin{tikzpicture}[baseline=-2.75pt,line width=.4pt,x=3em,y=2em]
\node[font=\footnotesize,inner sep=1.3pt,shape=circle] (L)  at (0,0) {\eqmakebox[vw]{$\mathbf i j$}};
\node[font=\footnotesize,inner sep=1.3pt,shape=circle] (R)  at (2,0) {\eqmakebox[vw]{$\mathbf i$}};
\path[->]
(R) edge node[font=\scriptsize,above=-.2ex] {$f^{\mathbf i}_j$} (L)
;
\path[->,line width=.4pt,font=\scriptsize, looseness=16, in=35, out=325,transform canvas={xshift=1pt,yshift=-.6pt}]
(R.340) edge (R.20)
;
\node[font=\scriptsize] at (2.94,0) {$...$};
\node[font=\scriptsize] at (-.99,0) {$...$};
\path[->,line width=.4pt,font=\scriptsize, looseness=16, in=40, out=320,overlay]
(R.320) edge (R.40)
;
\path[<-,line width=.4pt,font=\scriptsize, looseness=16, in=215, out=145,transform canvas={xshift=-1pt,yshift=-.6pt}]
(L.160) edge (L.200)
;
\path[<-,line width=.4pt,font=\scriptsize, looseness=15, in=220, out=140,overlay]
(L.140) edge (L.220)
;
\node[font=\scriptsize] at (-.7,-1) {$x^{\mathbf i j}_1{,} ..., x^{\mathbf i j}_{N_{\mathbf i j}}$};
\node[font=\scriptsize] at (2.7,-1) {$x^{\mathbf i}_1{,} ..., x^{\mathbf i}_{N_{\mathbf i}}$};
\end{tikzpicture}
\]
Now add relations $x^{\mathbf i}_r f^{\mathbf i}_j - f^{\mathbf i}_j X^{\mathbf i, j}_r$, where $X^{\mathbf i, j}_r$ is the image of the generator $x^{\mathbf i}_r$ of $\mathcal O_X (U_{\mathbf i})$ under the restriction map $\mathcal O_X (U_{\mathbf i}) \tikzto \mathcal O_X (U_{\mathbf i j})$, expressed in terms of a linear combination of irreducible paths (with respect to $R^{\mathbf i, j}$) in the generators $x^{\mathbf i j}_1, \dotsc, x^{\mathbf i j}_{N_{\mathbf i j}}$.

These relations are encoded into a reduction system by setting
\[
R^{\mathbf i, j} = \{ (x^{\mathbf i}_r f^{\mathbf i}_j, f^{\mathbf i}_j X^{\mathbf i, j}_r) \}_{1 \leq r \leq N_{\mathbf i}}
\]
and
\[
R_1 = \bigcup_{1 \leq \lvert \mathbf i \rvert \leq n - 1} R^{\mathbf i, j}.
\]

Finally, we define $R^{\mathrm{diag}} = R_0 \cup R_1 \cup R_2$ and let $J = (s - \phi_s)_{(s, \phi_s) \in R^{\mathrm{diag}}}$ as usual. By construction we have that $\Bbbk Q^{\mathrm{diag}} / J \simeq \mathcal O_X \vert_{\mathfrak U}!$.

Clearly $R^{\mathrm{diag}}$ is finite, so it only remains to show that $R^{\mathrm{diag}}$ is reduction-unique, so that $R^{\mathrm{diag}}$ satisfies $(\diamond)$ for $J$. Since $R_1$ has no overlaps and the overlaps in $R_0$ and $R_2$ resolve, we only need to show that overlaps in $R^{\mathrm{diag}}$ arising from combinations in $R_0, R_1, R_2$ also resolve. These additional overlaps are of the form
\begin{enumerate}
\item $s^{\mathbf i} f^{\mathbf i}_j$ for any $(s^{\mathbf i}, \phi_{s^{\mathbf i}}) \in R^{\mathbf i}$
\item $x^{\mathbf i}_r f^{\mathbf i}_k f^{\mathbf i k}_j$ for $j < k$ and any $1 \leq r \leq N_{\mathbf i}$.
\end{enumerate}
Note that at each vertex $\mathbf i$, reductions (with respect to $R^{\mathrm{diag}}$ or equivalently $R^{\mathbf i}$) induce the identity map on the quotient algebra $\Bbbk \langle x^{\mathbf i}_1, \dotsc, x^{\mathbf i}_{N_{\mathbf i}} \rangle / J^{\mathbf i} \simeq \mathcal O_X (U_{\mathbf i})$ and the arrows $f^{\mathbf i}_j$ correspond to the algebra homomorphisms $\mathcal O_X (U_{\mathbf i}) \tikzto \mathcal O_X (U_{\mathbf i j})$. The overlap $s^{\mathbf i} f^{\mathbf i}_j$ thus uniquely resolves to $f^{\mathbf i}_j Y$, where $Y$ is the unique linear combination of irreducible paths (with respect to $R^{\mathbf i j}$) such that $[Y] \in \Bbbk \langle x^{\mathbf i j}_1, \dotsc, x^{\mathbf i j}_{N_{\mathbf i j}} \rangle / J^{\mathbf i j} \simeq \mathcal O_X (U_{\mathbf i j})$ equals the image of $[s^{\mathbf i}] = [\phi_{s^{\mathbf i}}] \in \mathcal O_X (U_{\mathbf i})$ under $\mathcal O_X (U_{\mathbf i}) \tikzto \mathcal O_X (U_{\mathbf i j})$. Similarly, $x^{\mathbf i}_r f^{\mathbf i}_k f^{\mathbf i_k}_j$ uniquely resolves to $f^{\mathbf i}_j f^{\mathbf i j}_k Z$, where $Z$ is the unique linear combination of irreducible paths (with respect to $R^{\mathbf i j k}$) such that $[Z]$ is the image of $x^{\mathbf i}_r$ under  $\mathcal O_X (U_{\mathbf i}) \tikzto \mathcal O_X (U_{\mathbf i j k})$.
\end{proof}

\begin{theorem}
\label{theorem:diagramalgebrapqr}
Let $X$ be any separated Noetherian scheme and let $Q^{\mathrm{diag}}$ and $R^{\mathrm{diag}}$ be as in Proposition \ref{proposition:diagram}. Then the L$_\infty$ algebra $\mathfrak p (Q^{\mathrm{diag}}, R^{\mathrm{diag}})$ controls the deformation theory of $\Qcoh (X)$ as Abelian category.
\end{theorem}

\begin{proof}
It follows from Theorem \ref{theorem:BW1} and Proposition \ref{proposition:diagram} that $\mathfrak p(Q^{\mathrm{diag}}, R^{\mathrm{diag}})$ controls the deformation theory of the diagram algebra $\mathcal O_X \vert_{\mathfrak U}! \simeq \Bbbk Q^{\mathrm{diag}} / J$. The result then follows from Theorem \ref{theorem:equivalence}. 
\end{proof}

\begin{remark}
\label{remark:dinhlowen}
Dinh Van--Lowen \cite{dinhvanlowen} defined an L$_\infty$ algebra structure $\mathfrak{gs} (\mathcal O_X \vert_{\mathfrak U})$ on the Gersten\-haber--Schack complex by homotopy transfer from the DG Lie algebra $\mathfrak{h}(\mathcal O_X \vert_{\mathfrak U}!)$ (cf.~\S\ref{subsection:hochschild}). We thus have a zigzag of L$_\infty$ quasi-isomorphisms
\begin{align*}
\mathfrak p (Q^{\mathrm{diag}}, R^{\mathrm{diag}}) \tikzhookrightarrow \mathfrak{h}(\mathcal O_X \vert_{\mathfrak U}!)  \tikzhookleftarrow \mathfrak{gs}(\mathcal O_X \vert_{\mathfrak U})
\end{align*}
(cf.\ Remark \ref{remark:zigzag}) so that both $\mathfrak p (Q^{\mathrm{diag}}, R^{\mathrm{diag}})$ and $\mathfrak{gs} (\mathcal O_X \vert_{\mathfrak U})$ control the deformation theory of $\mathcal O_X \vert_{\mathfrak U}!$ (and thus the deformation theory of $\Qcoh (X)$).

Note however that $\mathfrak p (Q^{\mathrm{diag}}, R^{\mathrm{diag}})$ is practically always much smaller than $\mathfrak{gs}(\mathcal O_X \vert_{\mathfrak U})$. For instance, if $X = \mathbb A^d$ we may take $\mathfrak U = \{ X \}$ whence $\mathcal O_X \vert_{\mathfrak U}! \simeq \mathcal O_X (X) \simeq \Bbbk [x_1, \dotsc, x_d]$. The cochain complex underlying $\mathfrak p (Q^{\mathrm{diag}}, R^{\mathrm{diag}})$ has trivial differential and is therefore already isomorphic to $\HH^\hdot (X) \simeq \H^0 (\Lambda^\hdot \mathcal T_{\mathbb A^d})$ (see \cite[Lem.~10.21]{barmeierwang}), whereas the Gersten\-haber--Schack complex is the full Hochschild cochain complex for $\Bbbk [x_1, \dotsc, x_d]$. In this case deformations of $\Qcoh (\mathbb A^d)$ correspond to deformation quantizations of algebraic Poisson structures on $\mathbb A^d$ which is a difficult problem \cite{kontsevich1} and a formal solution does not necessarily give a handle on the problem of finding algebraizations. Working with $\mathfrak p (Q^{\mathrm{diag}}, R^{\mathrm{diag}})$ one can obtain explicit quantizations \cite[\S 10]{barmeierwang} which can be used to obtain (non-formal) strict deformation quantizations of algebraic Poisson structures \cite{barmeierschmitt}.

Now, if $X$ is a projective hypersurface, the cochain complex underlying $\mathfrak p(Q^{\mathrm{diag}}, R^{\mathrm{diag}})$ coincides with the BACH complex (see \S\ref{subsection:hypersurfaces}) at each vertex of the diagram instead of the Hochschild cochain complex. The BACH complex was also used by L.~Liu and W.~Lowen \cite{liulowen} to study the Hochschild cohomology of projective hypersurfaces $X$ which can be viewed as first-order deformations of $\Qcoh (X)$ up to equivalence. The L$_\infty$ algebra $\mathfrak p (Q^{\mathrm{diag}}, R^{\mathrm{diag}})$ thus gives the natural higher structure to use the BACH complex to study the full formal deformation theory of $\Qcoh (X)$ and moreover gives a route to study algebraizations.
\end{remark}

\subsubsection{Algebraization}
\label{subsubsection:algebraization}

Using the same idea as in Proposition \ref{proposition:Groebner}, we can use the admissible order on the algebra $\mathcal O_X (U)$ at each vertex of $Q^{\mathrm{diag}}$ to give a sufficient criterion for the existence of an algebraization of a formal deformation of the diagram algebra $\mathcal O_X \vert_{\mathfrak U}!$.

Recall that $R^{\mathrm{diag}} = R_0 \cup R_1 \cup R_2$ is obtained by gluing the commutative Gröbner basis with respect to some admissible order $\prec$ at each vertex. Let $S_0 = \{ s \}_{(s, \phi_s) \in R_0}$. Then we have the following result.

\begin{proposition}
\label{proposition:diagramalgebraization}
Let $\widetilde \phi \in \Hom (\Bbbk S, \Bbbk Q^{\mathrm{diag}} / J) \otimes (t)$ be a Maurer--Cartan element of the L$_\infty$ algebra $\mathfrak p(Q^{\mathrm{diag}}, R^{\mathrm{diag}}) \hatotimes (t)$ such that 
$\widetilde \phi(s) \prec s$ for each $s \in S_0$. Then $(\mathcal O_X \vert_{\mathfrak U}! \llrr{t}, \star_{\phi + \widetilde \phi})$ admits $(\mathcal O_X \vert_{\mathfrak U}! [t], \star_{\phi + \widetilde \phi})$ as an algebraization.
\end{proposition}

\begin{proof}
This follows from the same idea as Proposition \ref{proposition:Groebner}, namely if $\widetilde \phi (s)$ is polynomial in $t$ for each $s \in S$ and $\widetilde \phi$ satisfies the degree condition $\widetilde \phi (s) \prec s$ for all $s \in S_0$, then $\star_{\phi + \widetilde \phi}$ is a finite sum and is thus already well defined on $\mathcal O_X \vert_{\mathfrak U}! [t]$.
\end{proof}

This can be generalized to any complete local Noetherian $\Bbbk$-algebra $(B, \mathfrak m)$.

\subsubsection{An example in dimension $1$}

In the following example we show how to calculate the family of deformations of a (smooth) genus $3$ curve using the deformation theory of $\mathcal O_X \vert_{\mathfrak U}! \simeq \Bbbk Q^{\mathrm{diag}} / J$ for $Q^{\mathrm{diag}}$ and $J$ as in Proposition \ref{proposition:diagram}. Of course the deformation theory of smooth projective curves is one of the best understood examples, but we provide the example as a point of reference to illustrate which shape the deformation theory of $\Qcoh (X)$ takes when computed as deformations of $\mathcal O_X \vert_{\mathfrak U}! \simeq \Bbbk Q^{\mathrm{diag}} / J$. (Note that curves of genus $\geq 1$ do not admit tilting bundles, and so the usually more economical approach via tilting bundles studied in \S\ref{subsection:tilting} does not apply here.)

\begin{example}
\label{example:genus3}
{\it Deformations of a genus $3$ curve.}
Consider the smooth genus $3$ curve $X \subset \mathbb P^2 = \{ [x_0, x_1, x_2] \}$ cut out by the quartic equation
\[
F = x_0^3 x_1 + x_1^3 x_2 + x_2^4 = 0.
\]
Note that the point $[0, 0, 1]$ does not lie in $X$ so that $X \subset \mathbb P^2 \setminus \{ [0, 0, 1] \} = U_0 \cup U_1$ where
\begin{align*}
U_0 &= \{ [x_0, x_1, x_2] \mid x_0 \neq 0 \} = \{ [1, z, u] \} \simeq \Bbbk^2 \\
U_1 &= \{ [x_0, x_1, x_2] \mid x_1 \neq 0 \} = \{ [\zeta, 1, v] \} \simeq \Bbbk^2.
\end{align*}
Let $U, V$ denote the restrictions of $U_0, U_1$ to $X$. Setting $\mathfrak U = \{ U, V, U \cap V \}$ we can follow the proof of Proposition \ref{proposition:diagram} and write the diagram algebra $\mathcal O_X \vert_{\mathfrak U}!$ as the path algebra of the following quiver
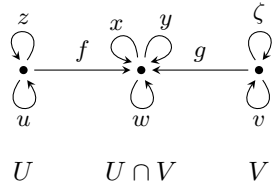
\[
\begin{tikzpicture}[baseline=-2.75pt,x=3.75em,y=1em]
\draw[line width=1pt, fill=black] (-1,0) circle(0.2ex);
\draw[line width=1pt, fill=black] (0,0) circle(0.2ex);
\draw[line width=1pt, fill=black] (1,0) circle(0.2ex);
\node[shape=circle, scale=0.7](L) at (-1,0) {};
\node[shape=circle, scale=0.7](M) at (0,0) {};
\node[shape=circle, scale=0.7](R) at (1,0) {};
\node[shape=circle, scale=0.9](LL) at (-1,0) {};
\path[->,line width=.4pt,font=\scriptsize, looseness=16, in=125, out=55]
(L.70) edge node[above=-.2ex] {$z$} (L.110)
;
\path[->,line width=.4pt,font=\scriptsize, looseness=16, in=305, out=235]
(L.250) edge node[below=-.2ex] {$u$} (L.290)
;
\path[->,line width=.4pt,font=\scriptsize, looseness=16, in=165, out=95, transform canvas={yshift=1pt}]
(M.110) edge node[above=-.2ex] {$x$} (M.150)
;
\path[<-,line width=.4pt,font=\scriptsize, looseness=16, in=85, out=15, transform canvas={yshift=1pt}]
(M.30) edge node[above=-.2ex] {$y$} (M.70)
;
\path[->,line width=.4pt,font=\scriptsize, looseness=16, in=305, out=235]
(M.250) edge node[below=-.2ex] {$w$} (M.290)
;
\path[<-,line width=.4pt,font=\scriptsize, looseness=16, in=125, out=55]
(R.70) edge node[above=-.2ex] {$\zeta$} (R.110)
;
\path[<-,line width=.4pt,font=\scriptsize, looseness=16, in=305, out=235]
(R.250) edge node[below=-.2ex] {$v$} (R.290)
;
\path[->,line width=.4pt,font=\scriptsize]
(L) edge node[above=-.2ex] {$f$} (M)
;
\path[->,line width=.4pt,font=\scriptsize]
(R) edge node[above=-.2ex] {$g$} (M)
;
\begin{scope}[yshift=-3.25em]
\node[font=\footnotesize] (U) at (-1,0) {$U$};
\node[font=\footnotesize] (UV) at (0,0) {$U \cap V$};
\node[font=\footnotesize] (V) at (1,0) {$V$};
\end{scope}
\end{tikzpicture}
\]
with ideal of relations $J$ generated by $s - \varphi_s$ for $(s, \varphi_s) \in R^{\mathrm{diag}}$ where $R^{\mathrm{diag}}$ is the reduction system consisting of the following pairs
\begin{itemize}[itemsep=.25em]
\item $(uz, zu)$, $(v \zeta, \zeta v)$, $(w x, xw)$, $(w y, y w) \phantom{^4}$ \hfill {\it commutativity of charts}
\item $(z f, f x)$, $(u f, f w)$, $(\zeta g, g y)$, $(v g, g y w)  \phantom{^4}$ \hfill {\it compatibility with morphisms}
\item $(xy, 1)$, $(yx, 1) \phantom{^4}$ \hfill {\it mutually inverse coordinates}
\item $(u^4, -z^3 u - z)$, $(w^4, -x^3 w - x)$, $(v^4, -\zeta^3 - v)$ \hfill {\it equation of curve}
\end{itemize}
which can be obtained from the admissible orders extending $z \prec u$, $\zeta \prec v$ and $x \prec y \prec w$ at the vertices $Q^{\mathrm{diag}}$ corresponding to the open sets $U, V, U \cap V$, respectively. Note that we have
\[
S = \{ uz, v \zeta, wx, wy, xy, yx, z f, u f, \zeta g, v g, u^4, v^4, w^4 \}.
\]
To show that $R^{\mathrm{diag}}$ satisfies ($\diamond$) for $J$  it suffices to show that the overlaps $u z f$, $u^4 f$, $v \zeta g$, $v^4 g$, $xyx$, $yxy$, $wxy$, $wyx$, $w^4 x$, $w^4 y$, $u^4 z$, $v^4 \zeta$ in $S_3$ are resolvable. For each overlap this is a short and straightforward computation, for example
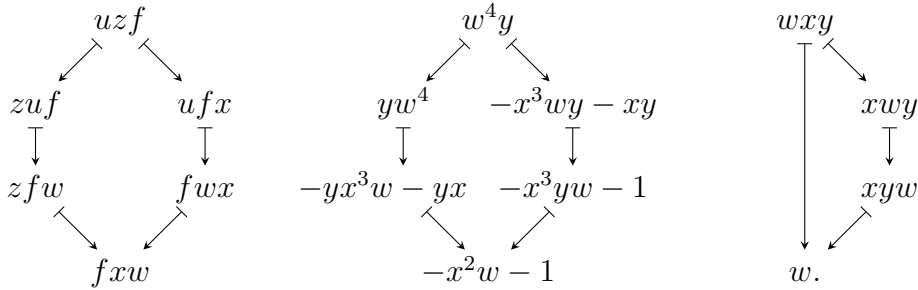
\[
\begin{tikzpicture}[baseline=-2.6pt,description/.style={fill=white,inner sep=1.75pt}]
\matrix (m) [matrix of math nodes, row sep={2.7em,between origins}, text height=1.5ex, column sep={2.7em,between origins}, text depth=0.25ex, ampersand replacement=\&]
{
\& u z f \& \\
z u f \&\& u f x \\
z f w \&\& f w x \\
\& f x w \& \\
};
\path[|->,line width=.4pt]
(m-1-2) edge (m-2-1)
(m-1-2) edge (m-2-3)
(m-2-1) edge (m-3-1)
(m-2-3) edge (m-3-3)
(m-3-1) edge (m-4-2)
(m-3-3) edge (m-4-2)
;
\end{tikzpicture}
\qquad
\begin{tikzpicture}[baseline=-2.6pt,description/.style={fill=white,inner sep=1.75pt}]
\matrix (m) [matrix of math nodes, row sep={2.7em,between origins}, text height=1.5ex, column sep={2.7em,between origins}, text depth=0.25ex, ampersand replacement=\&]
{
\& w^4 y \& \\
y w^4 \&\& -x^3 w y - xy \\
\mathllap{-y} x^3 w - y x \&\& -x^3 y w - 1 \\
\& -x^2 w - 1 \& \\
};
\path[|->,line width=.4pt]
(m-1-2) edge (m-2-1)
(m-1-2) edge (m-2-3)
(m-2-1) edge (m-3-1)
(m-2-3) edge (m-3-3)
(m-3-1) edge (m-4-2)
(m-3-3) edge (m-4-2)
;
\end{tikzpicture}
\quad
\begin{tikzpicture}[baseline=-2.6pt,description/.style={fill=white,inner sep=1.75pt}]
\matrix (m) [matrix of math nodes, row sep={2.7em,between origins}, text height=1.5ex, column sep={2.7em,between origins}, text depth=0.25ex, ampersand replacement=\&]
{
\& w x y \& \\
\&\& x w y \\
\&\& x y w \\
\& w. \& \\
};
\path[|->,line width=.4pt]
(m-1-2) edge (m-2-3)
(m-1-2) edge (m-4-2)
(m-2-3) edge (m-3-3)
(m-3-3) edge (m-4-2)
;
\end{tikzpicture}
\]

Now let $\widetilde \varphi \in \Hom (\Bbbk S, \Bbbk Q^{\mathrm{diag}} / J)$ be any $2$-cochain and as usual we denote $\widetilde \varphi_s = \widetilde \varphi (s) \in \Bbbk \Irr_S$. Let $F_U$, $F_V$ and $F_W$ denote the restrictions of $F$ to the charts $U_0$, $U_1$ and $U_{01} = U_0 \cap U_1$ of $\mathbb P^2$, i.e.\
\begin{align*}
F_U = u^4 + z^3 u + z, \quad F_V = v^4 + \zeta^3 + v,\quad \text{and} \quad F_W  = w^4 + x^3 w + x.
\end{align*}

In order for $\widetilde \varphi$ to satisfy the Maurer--Cartan equation, we need only check that \eqref{mc} holds for the elements in $S_3$ giving the conditions
\begin{equation}
\label{curvecocycle1}
\widetilde \varphi_{wx} = \widetilde \varphi_{wy} = \widetilde \varphi_{uz} = \widetilde \varphi_{v \zeta} = 0 \qquad \text{and} \qquad \widetilde \varphi_{xy} = \widetilde \varphi_{yx}
\end{equation}
as well as
\begin{equation}
\label{curvecocycle2}
\begin{aligned}
\widetilde \varphi_{v^4} g &  = g y^4 \widetilde \varphi_{w^4} + \frac{\partial F_V}{\partial \zeta} \widetilde \varphi_{\zeta g} +   \frac{\partial F_V}{\partial v} \widetilde \varphi_{v g} - g \frac{\partial F_W}{\partial x} y^3\widetilde \varphi_{yx}\\
\widetilde \varphi_{u^4} f & = f \widetilde \varphi_{w^4} + \frac{\partial F_U}{\partial u} \widetilde \varphi_{u f} + \frac{\partial F_U}{\partial z} \widetilde \varphi_{z f}.
\end{aligned}
\end{equation}
Here (\ref{curvecocycle1}) signifies that the commutativity relations of the charts should not be changed and (\ref{curvecocycle2}) imposes compatibility conditions between changing the morphisms and changing the equation of $X$ in the individual charts.

We may first deform the individual algebra $\Bbbk [\zeta, v] / (F_V) $ in the chart $V$ by setting
\[
\widetilde \varphi_{v^4} = \lambda_1 + \lambda_2 \zeta + \lambda_3 v + \lambda_4 \zeta v + \lambda_5 v^2 + \lambda_6 \zeta v^2
\]
since we have
\[
\Bbbk [\zeta, v] \bigm/ \bigl(\tfrac{\partial F_V}{\partial \zeta}, \tfrac{\partial F_V}{\partial v} \bigr) = \Bbbk [\zeta, v] / (4 v^3 + 1, 3 \zeta^2).
\]
Then setting $\widetilde \varphi_{z f} = \widetilde \varphi_{u f} = \widetilde \varphi_{\zeta g} = \widetilde \varphi_{v g} = \widetilde \varphi_{yx} = \widetilde \varphi_{xy} = 0$ and solving (\ref{curvecocycle2}) giving
\begin{align*}
\widetilde \varphi_{w^4} &= \eqmakebox[l1]{$\lambda_1 x^4$} + \eqmakebox[l2]{$\lambda_2 x^3$} + \eqmakebox[l3]{$\lambda_3 x^3 w$} + \eqmakebox[l4]{$\lambda_4 x^3 w$} + \eqmakebox[l5]{$\lambda_5 x^2 w^2$} + \eqmakebox[l6]{$\lambda_6 x w^2$} \\
\widetilde \varphi_{u^4} &= \eqmakebox[l1]{$\lambda_1 z^4$} + \eqmakebox[l2]{$\lambda_2 z^3$} + \eqmakebox[l3]{$\lambda_3 z^3 u$} + \eqmakebox[l4]{$\lambda_4 z^3 u$} + \eqmakebox[l5]{$\lambda_5 z^2 u^2$} + \eqmakebox[l6]{$\lambda_6 z u^2$}. 
\end{align*}
(Note that $\widetilde \phi$ satisfies $\widetilde \phi (s) \prec s$ for all $s \in S_0$, so we may view $\lambda_1, \dotsc, \lambda_6$ directly as ``actual'' rather than formal parameters, cf.\ Proposition \ref{proposition:diagramalgebraization}.)

Note that for a genus $g$ curve $C$ one has $\HH^2 (C) \simeq \H^1 (\mathcal T_C) \simeq \Bbbk^{3g-3}$. Here $X$ is a genus $3$ curve and indeed we obtain a $6$-dimensional family of nontrivial ``actual'' deformations of $X$ parametrized by $\lambda = (\lambda_1, \dotsc, \lambda_6) \in \Bbbk^6$ with fibres $X_\lambda = \{ F_\lambda = 0 \} \subset \mathbb P^2$, where 
$$
F_{\lambda} = x_0^3 x_1 + x_2^4 + \lambda_1 x_1^4 + \lambda_2 x_0 x_1^3 + (1 + \lambda_3) x_1^3 x_2 + \lambda_4 x_0 x_1^2 x_2 + \lambda_5 x_1^2 x_2^2 + \lambda_6 x_0 x_1 x_2^2.
$$
\end{example}

This example shows that the deformation theory of $\Qcoh (X)$ can be studied rather explicitly. (See \S\ref{subsection:zkdiagram} for the case $X =  \Tot \mathcal O_{\mathbb P^1} (-k)$.) Indeed, each element $(s, \varphi_s)$ in the reduction system has a clear geometric meaning: it corresponds either to a commutativity relation of the local coordinates, to the identification of coordinates across charts, to the defining equations, or to the commutativity of the coordinate changes across charts. That is, the geometric meaning of the modifications to the reduction system is very much transparent.

Note that in general, a deformation of $\mathcal O_X \vert_{\mathfrak U}$ may be a not necessarily commutative diagram of not necessarily commutative associative algebras, so it is natural to look at deformations of $\mathcal O_X \vert_{\mathfrak U}!$ in this general context of path algebras of quivers with relations.

\subsection{Deformations via tilting bundles}
\label{subsection:tilting}

By Theorem \ref{theorem:equivalence} deformations of the Abelian category $\Qcoh (X)$ admit a description as deformations of the diagram algebra $\mathcal O_X \vert_{\mathfrak U}!$ which by Proposition \ref{proposition:diagram} can be written as $\Bbbk Q^{\mathrm{diag}} / J$.

In case the variety or scheme admits a tilting bundle (e.g.\ projective spaces \cite{beilinson}, Grassmannians, quadrics \cite{kapranov1,kapranov2}, rational surfaces \cite{hilleperling}, hypertoric varieties \cite{spenkovandenbergh}) --- for example the tilting bundle obtained from a strong full exceptional collection --- there is a much more economical description of this deformation theory by means of a smaller quiver: we have
\[
\End \mathcal E \simeq \Bbbk Q^{\mathrm{tilt}} / I
\]
for some finite quiver $Q^{\mathrm{tilt}}$ and some ideal of relations $I$. Here $Q^{\mathrm{tilt}}$ is constructed by putting a vertex for each direct summand of $\mathcal E$ and adding arrows to generate the morphism spaces between the direct summands.

The tilting bundle $\mathcal E$ induces a {\it derived} equivalence between the Abelian category of quasi-coherent sheaves on $X$ and the Abelian category of right modules for the endomorphism algebra $\Bbbk Q^{\mathrm{tilt}} / I \simeq \End \mathcal E$
\begin{align}
\label{derivedequivalence}
\D (\Mod (\End \mathcal E)) &\simeq \D (\Qcoh (X)) \\
\intertext{given by the functors $\mathbf R \mathrm{Hom} (\mathcal E{,} \blank)$ and $\blank \otimes^{\mathbf L} \mathcal E$ and (\ref{derivedequivalence}) induces an isomorphism of Hochschild cohomologies}
\HH^\hdot (\End \mathcal E) &\simeq \HH^\hdot (X) \notag.
\end{align}
(See for example \cite{baer,bondal,buchweitzhille,hillevandenbergh}.)

\begin{theorem}
\label{theorem:tiltingbundlepqr}
Let $X$ be a separated Noetherian scheme and let $\mathcal E$ be a tilting bundle on $X$. Write $\End \mathcal E = \Bbbk Q^{\mathrm{tilt}} / I$ and let $R^{\mathrm{tilt}}$ be any reduction system satisfying \textup{($\diamond$)} for $I$. Then $\mathfrak p (Q^{\mathrm{tilt}}, R^{\mathrm{tilt}})$ controls the deformation theory of the Abelian category $\Qcoh (X)$.
\end{theorem}

\begin{proof}
It follows from Hille--Van den Bergh \cite[Thm.~7.6]{hillevandenbergh} that the derived tensor functor $\blank \otimes^{\mathbf L} \mathcal E$ induces an equivalence between $\D (\Mod (\End \mathcal E))$ and $\D (\Qcoh (X))$. Then by Lowen--Van den Bergh \cite[Thm.~6.1]{lowenvandenbergh1} there is a B$_\infty$ quasi-isomorphism between $\mathrm C^\hdot (\End \mathcal E, \End \mathcal E)$ and $\mathrm C^\hdot_{\mathrm{Ab}} (\Qcoh (X))$ which in particular gives an L$_\infty$ quasi-isomorphism between $\mathfrak h (\End \mathcal E)$ and $\mathfrak{lv} (\Qcoh (X))$, where the latter denotes the DG Lie algebra structure on $\mathrm C^{\hdot + 1}_{\mathrm{Ab}} (\Qcoh (X))$ controlling the deformation theory of $\Qcoh (X)$ introduced in \cite{lowenvandenbergh1}. It now follows from the L$_\infty$ quasi-isomorphism in Theorem \ref{theorem:BW1} that $\mathfrak p (Q^{\mathrm{tilt}}, R^{\mathrm{tilt}})$ controls the deformation theory of $\Qcoh (X)$.
\end{proof}

\begin{remark}
\label{remark:zigzag}
Collecting the various DG Lie and L$_\infty$ algebras which all control the deformation theory of $\Qcoh (X)$ in some way, we have the following natural zigzags of L$_\infty$ quasi-isomorphisms 
\begin{equation}
\label{zigzag}
\begin{tikzpicture}[baseline=-2.6pt,description/.style={fill=white,inner sep=1.75pt}]
\matrix (m) [matrix of math nodes, row sep=0em, text height=1.5ex, column sep=1.5em, row sep=1.5em, text depth=0.25ex, ampersand replacement=\&]
{
\& \mathfrak{gs} (\mathcal O_X \vert_{\mathfrak U}) \& \mathfrak{lv} (\Qcoh (X)) \\
\mathfrak p (Q^{\mathrm{diag}}, R^{\mathrm{diag}}) \& \mathfrak h (\mathcal O_X \vert_{\mathfrak U}!) \& \mathfrak{lv} (\Mod (\mathcal O_X \vert_{\mathfrak U}!)) \\
\mathfrak p (Q^{\mathrm{tilt}}, R^{\mathrm{tilt}}) \& \mathfrak h (\End \mathcal E) \& \mathfrak{lv} (\Mod (\End \mathcal E)) \\
};
\path[->,line width=.4pt,font=\scriptsize]
(m-1-2) edge (m-2-2)
(m-1-3) edge (m-2-3)
(m-2-1) edge (m-2-2)
(m-2-2) edge (m-2-3)
(m-2-3) edge (m-3-3)
(m-3-1) edge (m-3-2)
(m-3-2) edge (m-3-3)
;
\end{tikzpicture}
\end{equation}
where the two left horizontal arrows are given in Theorem \ref{theorem:BW1}, the middle vertical arrow was constructed by Dinh Van--Lowen \cite{dinhvanlowen} (cf.\ Remark \ref{remark:dinhlowen}) and the other arrows follow from Lowen--Van den Bergh \cite{lowenvandenbergh1}.

The advantage of the two L$_\infty$ algebras $\mathfrak p (Q^*, R^*)$ for $\ast \in \{ \mathrm{diag}, \mathrm{tilt} \}$ on the left-hand side is that their underlying complexes are usually much smaller. These L$_\infty$ algebras control the full deformation theory of $\Qcoh (X)$, but at the same time the deformations are often possible to construct explicitly (even by hand) by using the combinatorial criterion for the Maurer--Cartan equation given in Theorem \ref{theorem:higher-brackets}.
\end{remark}

To complete the picture, we note that the derived equivalence \eqref{derivedequivalence} obtained by tilting can be extended to any formal deformation.

\begin{proposition}
\label{proposition:derived}
Let $X$ be any separated Noetherian scheme, let $\mathfrak U$ be an affine open cover of $X$ closed under intersections and let $\mathcal E$ be a tilting bundle on $X$. Then the derived equivalence $\D (\Qcoh (X)) \simeq \D (\Mod (\End \mathcal E))$ lifts to any formal deformation.

More precisely, let $(B, \mathfrak m)$ be any complete local Noetherian $\Bbbk$-algebra, let $\widetilde \Phi$ be a Maurer--Cartan element of $\mathfrak p (Q^{\mathrm{diag}}, R^{\mathrm{diag}}) \hatotimes \mathfrak m$ and let $\widetilde \phi$ be the corresponding Maurer--Cartan element of $\mathfrak p (Q^{\mathrm{tilt}}, R^{\mathrm{tilt}}) \hatotimes \mathfrak m$. Then we have a triangulated equivalence
\[
\D (\Qcoh ((\mathcal O_X \vert_{\mathfrak U}! \hatotimes B, \star_{\Phi + \widetilde \Phi}))) \simeq \D (\Mod ((\End \mathcal E \hatotimes B, \star_{\phi + \widetilde \phi}))).
\]
\end{proposition}

\begin{proof}
The tilting bundle $\mathcal E$ is a compact generator of $\D (\Qcoh (X))$ and Blanc--Katzar\-kov--Pandit \cite[Thm.~4.29]{blanckatzarkovpandit} showed more generally that for any $\Bbbk$-linear $\infty$-category $\mathcal D$, a compact generator of $\mathcal D$ lifts to any formal deformation of $\mathcal D$. A proof in the case of purely commutative deformations can also be found in Karmazyn \cite[Thm.~3.4]{karmazyn}.
\end{proof}

\section{A quasi-projective example}
\label{section:totalspaces}

In this section we give a detailed description of the deformation theory of the Abelian category $\Qcoh (Z_k)$, where $Z_k = \Tot \mathcal O_{\mathbb P^1} (-k)$ for $k \geq 2$ is the total space of a negative line bundle on $\mathbb P^1$ with first Chern class $-k$. This family of smooth quasi-projective surfaces allows us to illustrate the various aspects of the combinatorial approach to deformations of Abelian categories of (quasi)coherent sheaves, such as the obstruction calculus for $\mathfrak p (Q, R)$ and the problem of finding algebraizations. The surfaces $Z_k$ are particularly well-suited for the following reasons.

Firstly, the surfaces $Z_k$ admit tilting bundles, so that we can compare the two different descriptions of the deformation theory of $\Qcoh (Z_k)$ given in \S\ref{section:applications} via deformations of the diagram algebra and via deformations of the endomorphism algebra of the tilting bundle. Since $Z_k$ is covered by two copies of $\mathbb A^2$, both the diagram algebra and the endomorphism algebra of the tilting bundle can be easily described.

Secondly, the Hochschild--Kostant--Rosenberg theorem decomposes the second Hochschild cohomology into two direct summands
\[
\HH^2 (Z_k) \simeq \H^0 (\Lambda^2 \mathcal T_{Z_k}) \oplus \H^1 (\mathcal T_{Z_k})
\]
corresponding to noncommutative deformations parametrized by algebraic Poisson structures and commutative deformations, respectively (see Lemma \ref{lemma:hochschildminres}). Moreover, obstructions lie in
\[
\HH^3 (Z_k) \simeq \H^1 (\Lambda^2 \mathcal T_{Z_k}) \simeq \Bbbk^{k-3}
\]
which is nontrivial for $k \geq 4$ allowing us to illustrate a nontrivial obstruction calculus of the L$_\infty$ algebra $\mathfrak p (Q, R)$ by using the necessary and sufficient condition given in Theorem \ref{theorem:higher-brackets}. 

Thirdly, the surfaces $Z_k$ are minimal resolutions of the cyclic surface singularities $X_k$ of type $\frac{1}{k} (1, 1)$ and any deformation of $\Qcoh (Z_k)$ induces a deformation of $X_k$. In \S\ref{section:singularities} we use the geometric description of the deformation theory of $\Qcoh (Z_k)$ to produce a family of simultaneously commutative and noncommutative deformations of $X_k$.

\subsection{Deformations of the diagram algebra}
\label{subsection:zkdiagram}

The surface $Z_k$ is a smooth toric surface covered by only two affine open sets $U = \{ (z, u) \in \Bbbk^2 \}$ and $V = \{ (\zeta, v) \in \Bbbk^2 \}$ and for $z, \zeta \neq 0$ we have $(\zeta, v) = (z^{-1}, z^k u)$ on $U \cap V \simeq \Bbbk^\times \times \Bbbk$. (Here $z$ and $\zeta$ are the local coordinates on $\mathbb P^1$.) We may thus take $\mathfrak U = \{ U, V, U \cap V \}$ as the affine open cover closed under intersections and consider the diagram of algebras $\mathcal O_{Z_k} \vert_{\mathfrak U}$ given by
\[
\mathcal O_{Z_k} (U) \tikzlongrightarrow \mathcal O_{Z_k} (U \cap V) \tikzlongleftarrow \mathcal O_{Z_k} (V)
\]
which in the above coordinates can be written as
\begin{equation}
\label{diagramzk}
\begin{tikzpicture}[baseline=-2.6pt,description/.style={fill=white,inner sep=1.75pt}]
\matrix (m) [matrix of math nodes, row sep=0em, text height=1.5ex, column sep=1.5em, text depth=0.25ex, ampersand replacement=\&]
{
\Bbbk [z, u] \& \Bbbk [x, y, w] / (xy - 1) \& \Bbbk [\zeta, v] \\[.5em]
z \& \mathrlap{x} \phantom{\Bbbk [x, y, w] / (xy - 1)} \mathllap{y} \& \zeta \\
u \& \mathrlap{w} \phantom{\Bbbk [x, y, w] / (xy - 1)} \mathllap{x^k w} \& v\rlap{.} \\
};
\path[->,line width=.4pt,font=\scriptsize]
(m-1-1) edge node[above=-.2ex] {$f$} (m-1-2)
(m-1-3) edge node[above=-.2ex] {$g$} (m-1-2)
;
\path[|->,line width=.4pt]
(m-2-1) edge (m-2-2)
(m-2-3) edge (m-2-2)
(m-3-1) edge (m-3-2)
(m-3-3) edge (m-3-2)
;
\end{tikzpicture}
\end{equation}
By Proposition \ref{proposition:diagram} have an algebra isomorphism $\mathcal O_{Z_k} \vert_{\mathfrak U}! \simeq \Bbbk Q^{\mathrm{diag}} / J$, where $Q^{\mathrm{diag}}$ is the quiver
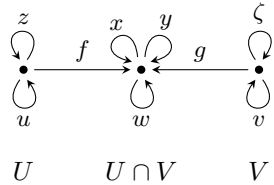
\[
\begin{tikzpicture}[baseline=-2.75pt,x=3.75em,y=1em]
\draw[line width=1pt, fill=black] (-1,0) circle(0.2ex);
\draw[line width=1pt, fill=black] (0,0) circle(0.2ex);
\draw[line width=1pt, fill=black] (1,0) circle(0.2ex);
\node[shape=circle, scale=0.7](L) at (-1,0) {};
\node[shape=circle, scale=0.7](M) at (0,0) {};
\node[shape=circle, scale=0.7](R) at (1,0) {};
\node[shape=circle, scale=0.9](LL) at (-1,0) {};
\path[->,line width=.4pt,font=\scriptsize, looseness=16, in=125, out=55]
(L.70) edge node[above=-.2ex] {$z$} (L.110)
;
\path[->,line width=.4pt,font=\scriptsize, looseness=16, in=305, out=235]
(L.250) edge node[below=-.2ex] {$u$} (L.290)
;
\path[->,line width=.4pt,font=\scriptsize, looseness=16, in=165, out=95, transform canvas={yshift=1pt}]
(M.110) edge node[above=-.2ex] {$x$} (M.150)
;
\path[<-,line width=.4pt,font=\scriptsize, looseness=16, in=85, out=15, transform canvas={yshift=1pt}]
(M.30) edge node[above=-.2ex] {$y$} (M.70)
;
\path[->,line width=.4pt,font=\scriptsize, looseness=16, in=305, out=235]
(M.250) edge node[below=-.2ex] {$w$} (M.290)
;
\path[<-,line width=.4pt,font=\scriptsize, looseness=16, in=125, out=55]
(R.70) edge node[above=-.2ex] {$\zeta$} (R.110)
;
\path[<-,line width=.4pt,font=\scriptsize, looseness=16, in=305, out=235]
(R.250) edge node[below=-.2ex] {$v$} (R.290)
;
\path[->,line width=.4pt,font=\scriptsize]
(L) edge node[above=-.2ex] {$f$} (M)
;
\path[->,line width=.4pt,font=\scriptsize]
(R) edge node[above=-.2ex] {$g$} (M)
;
\begin{scope}[yshift=-3.25em]
\node[font=\footnotesize] (U) at (-1,0) {$U$};
\node[font=\footnotesize] (UV) at (0,0) {$U \cap V$};
\node[font=\footnotesize] (V) at (1,0) {$V$};
\end{scope}
\end{tikzpicture}
\]
and $J$ is the two-sided ideal of relations generated by $s - \Phi_s$ for $(s, \Phi_s) \in R^{\mathrm{diag}}$ where $R^{\mathrm{diag}}$ is the reduction system consisting of the following pairs
\begin{itemize}[itemsep=.25em]
\item $(u z, z u)$, $(v \zeta, \zeta v)$, $(w x, x w)$, $(w y, y w)$ \hfill {\it commutativity of charts}
\item $(z f, f x)$, $(u f, f w)$, $(\zeta g, g y)$, $(v g, g x^k w)$ \hfill {\it compatibility with morphisms}
\item $(x y, 1)$, $(y x, 1)$ \hfill {\it mutually inverse coordinates.}
\end{itemize}
In particular, we have that
\begin{equation}\label{reductionSzk}
\begin{aligned}
S & = \{ u z, v \zeta, w x, w y, z f, u f, \zeta g, v g, x y, y x \}\\
S_3 & = \{uzf, v \zeta g, wxy, wyx, xyx, yxy\}.
\end{aligned}
\end{equation}
We denote by $\Phi \in \Hom (\Bbbk S, \Bbbk Q^{\mathrm{diag}} / J)$ the element determining $R^{\mathrm{diag}}$ (cf.\ Remark \ref{remark:f}).

The reduction system $R^{\mathrm{diag}}$ can be obtained from the construction in the proof of Proposition \ref{proposition:diagram} by choosing the Gröbner bases on $\mathcal O_{Z_k} (U), \mathcal O_{Z_k} (V), \mathcal O_{Z_k} (U \cap V)$ corresponding to the admissible orders extending $z \prec u$, $\zeta \prec v$ and $x \prec y \prec w$, respectively.

\begin{remark}
The quiver for $Z_k = \Tot \mathcal O_{\mathbb P^1} (-k)$ is exactly the same as the quiver in Example \ref{example:genus3}. However the relations are different: the genus $3$ curve of Example \ref{example:genus3} was written as a hypersurface in $\Tot \mathcal O_{\mathbb P^1} (1) = U_0 \cup U_1$, where $U_0, U_1$ are two of the standard affine coordinate charts of $\mathbb P^2$.
\end{remark}

From the reduction system $R^{\mathrm{diag}}$ we obtain a projective resolution $P_\ldot$ \eqref{resolution} of the diagram algebra $\mathcal O_{Z_k} \vert_{\mathfrak U}!$. The following lemma can be calculated geometrically as Čech cohomology with respect to the cover $\{ U, V \}$ \cite{barmeier,barmeiergasparim1,barmeiergasparim2} or algebraically using \eqref{cohomologies} and Remark \ref{remark:computinghh2}.

\begin{lemma}
\label{lemma:hochschildminres}
The cohomology groups relevant to the deformation theory appearing in the Hoch\-schild--Kostant--Rosenberg decomposition of $\HH^i (Z_k)$ are the following
\begin{flalign*}
&& \H^0 (\mathcal O_{Z_k}) &\simeq \Bbbk [u, zu, \dotsc, z^k u] \simeq \Bbbk [z_0, \dotsc, z_k] / (z_i z_{j+1} - z_{i+1} z_j)_{0 \leq i < j < k} && \\[.5em]
&& \H^1 (\mathcal T_{Z_k}) &\simeq \Bbbk \bigl\langle z^{-k+j} \tfrac{\partial}{\partial u} \bigr\rangle_{1 \leq j \leq k-1} \simeq \Bbbk^{k-1} && \\[.5em]
&& \H^0 (\Lambda^2 \mathcal T_{Z_k}) &\simeq
\begin{cases}
\H^0 (\mathcal O_{Z_k}) \bigl\langle \tfrac{\partial}{\partial z} {\hair \wedge \hair} \tfrac{\partial}{\partial u} \bigr\rangle & \textup{if $k = 2$} \\[.3em]
\H^0 (\mathcal O_{Z_k}) \bigl\langle z^j u \tfrac{\partial}{\partial z} {\hair \wedge \hair} \tfrac{\partial}{\partial u} \bigr\rangle_{0 \leq j \leq 2} & \textup{if $k \geq 3$}
\end{cases} \\[.5em]
&& \H^1 (\Lambda^2 \mathcal T_{Z_k}) &\simeq
\begin{cases}
0 & \textup{if $k = 2, 3$} \\[.3em]
\Bbbk \bigl\langle z^{-j} \tfrac{\partial}{\partial z} {\hair \wedge \hair} \tfrac{\partial}{\partial u} \bigr\rangle_{1 \leq j \leq k-3} \simeq \Bbbk^{k-3} & \textup{if $k \geq 4$}.
\end{cases}
\end{flalign*}
\end{lemma}

Here we express these cohomology groups in terms of the polyvector fields in the chart $U$, which can be naturally identified with cocycles in $\mathfrak p(Q^{\mathrm{diag}}, R^{\mathrm{diag}})$. For instance, the basis element $z^{-k+j} \tfrac{\partial}{\partial u}$ of $\H^1 (\mathcal T_{Z_k})$ corresponds to the $2$-cocycle $\widetilde \Phi \in \Hom(\Bbbk S, \Bbbk Q^{\mathrm{diag}} / J)$ defined by
\[
\widetilde \Phi(uf) = f y^{k-j}
\]
and $\widetilde \Phi(s)= 0$ for any $s \in S \setminus \{uf \}$. Similarly, the generator $ z^ju \tfrac{\partial}{\partial z} {\hair \wedge \hair} \tfrac{\partial}{\partial u}$ of $\H^0 (\Lambda^2 \mathcal T_{Z_k})$ corresponds to the cocycle $\widetilde \Phi \in \Hom(\Bbbk S, \Bbbk Q^{\mathrm{diag}} / J)$ given by 
\begin{align}
\label{align:cocyclezk}
\widetilde \Phi(uz) = z^ju,\quad \widetilde \Phi(v \zeta) = -\zeta^{2-j} v,\quad \widetilde \Phi(wx)=x^jw,\quad \widetilde \Phi(wy) = -y^{2-j}w
\end{align}
and $\widetilde \Phi(s) = 0$ for all other $s \in S$.

By Theorems \ref{theorem:BW1} and \ref{theorem:equivalence} the deformation theory of $\Qcoh(Z_k)$ is controlled by the L$_\infty$ algebra $\mathfrak p(Q^{\mathrm{diag}}, R^{\mathrm{diag}})$. In the following we will use Theorem \ref{theorem:higher-brackets} to construct a family of explicit formal deformations of $\Qcoh(Z_k)$ and use \S\ref{subsubsection:algebraization} to find algebraizations where possible.
 
\subsubsection{Commutative deformations}
\label{subsubsection:zkcommutativediagram}

The commutative deformations of the surfaces $Z_k$ were studied in \cite{barmeiergasparim1}, where it was shown that the nontrivial deformations corresponding to cocycles in $\H^1 (\mathcal T_{Z_k})$ are smooth affine varieties. Here we illustrate how to study these ``commutative'' deformations of $\Qcoh (Z_k)$ via the diagram algebra $\mathcal O_{Z_k} \vert_{\mathfrak U}! \simeq \Bbbk Q^{\mathrm{diag}} / J$. 

Consider the following element $\widetilde \Phi \in \Hom (\Bbbk S, \Bbbk Q^{\mathrm{diag}} / J) \hatotimes (t_1, \dotsc, t_{k-1})$ given by 
\[
\widetilde \Phi(uf) =  \sum_{j =1}^{k-1} fy^{j} t_{j}
\]
and $\widetilde \Phi(s) = 0$ for all other $s \in S$ in \eqref{reductionSzk}. By Theorem \ref{theorem:higher-brackets}, $\widetilde \Phi$ is a Maurer--Cartan element of $\mathfrak p(Q^{\mathrm{diag}}, R^{\mathrm{diag}})  \hatotimes (t_1, \dotsc, t_{k-1})$ since $\star_{\Phi + \widetilde \Phi}$ is associative on elements in $S_3$ in \eqref{reductionSzk} as can be easily checked (cf.\ \eqref{check} below). This corresponds to a commutative deformation of $Z_k$ and this deformation is algebraizable (cf.\ \S\ref{subsubsection:algebraization}). Evaluating the algebraization to $t_i \tikzmapsto \lambda_i$ for some $\lambda = (\lambda_1, \dotsc, \lambda_{k-1}) \in \Bbbk^{k-1}$, we obtain the diagram algebra of the commutative deformation of $Z_k$ given by the diagram
\begin{equation}
\label{zkcomm}
\begin{tikzpicture}[baseline=-2.6pt,description/.style={fill=white,inner sep=1.75pt}]
\matrix (m) [matrix of math nodes, row sep=0em, text height=1.5ex, column sep=1.5em, text depth=0.25ex, ampersand replacement=\&]
{
\Bbbk [z, u] \& \,\Bbbk [x, y, w] \,/\, (xy - 1)\, \& \Bbbk [\zeta, v] \\[.5em]
z \& \mathrlap{x} \phantom{\,\Bbbk [x, y, w] \,/\, (xy - 1)\,} \mathllap{y} \& \zeta \\
u \& \mathrlap{w + \sum\limits_{j =1}^{k-1}\lambda_j fy^{j} } \phantom{\,\Bbbk [x, y, w] \,/\, (xy - 1)\,} \mathllap{x^k w} \& v\rlap{.} \\
};
\path[->,line width=.4pt,font=\scriptsize]
(m-1-1) edge node[above=-.2ex] {$f$} (m-1-2)
(m-1-3) edge node[above=-.2ex] {$g$} (m-1-2)
;
\path[|->,line width=.4pt]
(m-2-1) edge (m-2-2)
(m-2-3) edge (m-2-2)
(m-3-1) edge (m-3-2)
(m-3-3) edge (m-3-2)
;
\end{tikzpicture}
\end{equation}
This is a diagram of commutative algebras and the Abelian category of quasi-coherent modules for the diagram \eqref{zkcomm} (see Definition \ref{definition:modules}) is equivalent to the category of quasi-coherent sheaves on the commutative deformation of $Z_k$. (This is an analogue of the usual equivalence between $\Qcoh (\Spec A) \simeq \Mod (A)$ in the affine case.)

\subsubsection{Noncommutative deformations}
\label{subsubsection:diagramnoncommutative}

The ``purely noncommutative'' deformations of $Z_k$ corresponding to quantizations of Poisson structures on $Z_k$ were studied (over $\mathbb C$) in \cite{barmeiergasparim2} by using Kontsevich's universal quantization formula on $\mathbb C^2$. Here we give a different combinatorial construction of the quantizations via $\mathfrak p (Q^{\mathrm{diag}}, R^{\mathrm{diag}})$ that also allows us to construct a ``$q$-deformation'' of $Z_k$.

Consider the following element $\widetilde \Phi  \in \Hom (\Bbbk S, \Bbbk Q^{\mathrm{diag}} / J) \hatotimes (t_1', t_2')$ given by 
\[
\widetilde \Phi(uz) = \alpha(z)u,\quad  \widetilde \Phi(v\zeta) = \beta(\zeta) v,\quad  \widetilde \Phi(wx) = \alpha(x)w,\quad  \widetilde \Phi(wy) = \beta(y) w
\] 
and $\widetilde \Phi(s) = 0$ for all other $s \in S$ in \eqref{reductionSzk}. Here $\alpha(z) =  t_1'+zt_2'$ and $\beta(\zeta) = \sum_{i=1}^{\infty} \zeta (-\zeta t_1' - t_2')^i$ so that $(x+\alpha(x)) (y + \beta(y)) = 1$. By Theorem \ref{theorem:higher-brackets} we may check that $\widetilde \Phi$ is a Maurer--Cartan element of $\mathfrak p(Q^{\mathrm{diag}}, R^{\mathrm{diag}}) \hatotimes (t_1', t_2')$.

Concretely, this can be done by checking \eqref{mc} on elements of $S_3$ in \eqref{reductionSzk}. For example, for $uzf \in S_3$ we have $u \star (z \star f) = (u \star z ) \star f$ as illustrated by the following diagram
\begin{equation}
\label{check}
\begin{tikzpicture}[baseline=-2.6pt,description/.style={fill=white,inner sep=1.75pt}]
\matrix (m) [matrix of math nodes, row sep={2.7em,between origins}, text height=1.5ex, column sep={2.7em,between origins}, text depth=0.25ex, ampersand replacement=\&]
{
\& u z f \& \\
ufx \&\& zuf + \alpha(z) u f \\
fwx \&\& z fw + \alpha(z) fw \\
\& f x w + f \alpha(x) w\rlap{.} \& \\
};
\path[|->,line width=.4pt]
(m-1-2) edge (m-2-1)
(m-1-2) edge (m-2-3)
(m-2-1) edge (m-3-1)
(m-2-3) edge (m-3-3)
(m-3-1) edge (m-4-2)
(m-3-3) edge (m-4-2)
;
\end{tikzpicture}
\end{equation}

Let us take $t_1' = \mu_0 \hbar$ and $t_2' = \mu_1 \hbar$, where $\mu_0, \mu_1 \in \Bbbk$ and $\hbar$ is another formal parameter. The corresponding deformation of the diagram algebra $\mathcal O_{Z_k} \vert_{\mathfrak U}! \simeq \Bbbk Q^{\mathrm{diag}} / J$ gives a deformation quantization of the Poisson structure $(\mu_0 u + \mu_1 z u) \tfrac{\partial}{\partial z} {\hair \wedge \hair} \tfrac{\partial}{\partial u}$ of $Z_k$, which can be seen by comparing the first-order term of $\widetilde \Phi$ with the cocycle defined in \eqref{align:cocyclezk}. The resulting formal deformation is the diagram algebra of the following diagram of algebras
\begin{equation}
\label{diagramnoncommutative}
(\Bbbk [z, u] \llrr{\hbar}, \star) \toarg{f} (\Bbbk [x, y, w] \llrr{\hbar} / (xy - 1), \star) \leftarrowarg{g} (\Bbbk [\zeta, v] \llrr{\hbar}, \star).
\end{equation}
Similarly as for the commutative case, the Abelian category of quasi-coherent modules for the diagram \eqref{diagramnoncommutative} can be viewed as the analogue of the Abelian category of quasi-coherent sheaves on some (formal) ``noncommutative deformation'' of $Z_k$.

\begin{remark}
\label{remark:graphical}
In \cite[\S 10]{barmeierwang} we gave a graphical description of the combinatorial star product $\star$, similar to Kontsevich's graphical calculus for his universal quantization formula \cite{kontsevich1}. In particular, the star product on the individual algebras in \eqref{diagramnoncommutative} can be given by the formula
\[
a \star b = \sum_{n \geq 0} \hbar^n \sum_{\Pi \in \mathfrak G_{n, 2}} C_\Pi (a, b)
\]
where $\mathfrak G_{n,2}$ is a set of admissible graphs for the combinatorial star product and $C_\Pi$ is a bidifferential operator associated to an admissible graph $\Pi$. See also \cite{barmeierschmitt} for further applications in deformation quantization.
\end{remark}

\paragraph{\it A $q$-deformation of $Z_k$.} 
From the diagram \eqref{diagramnoncommutative} we can also get an ``actual'' noncommutative deformation of $\Qcoh (Z_k)$ corresponding to the quantization of the ``log-canonical'' Poisson structure $z u \tfrac{\partial}{\partial z} {\hair \wedge \hair} \tfrac{\partial}{\partial u}$. Let $\mathcal A^q$ denote the diagram of algebras
\begin{equation}
\label{diagramq}
\Bbbk \langle z, u \rangle / (u z - q z u) \toarg{f} \Bbbk \langle x, y, w \rangle / J \leftarrowarg{g} \Bbbk \langle \zeta, v \rangle / (v \zeta - q^{-1} \zeta v)
\end{equation}
where $J = (w x - q x w, w y - q^{-1} y w, x y - 1, y x - 1)$ and $q \in \Bbbk \setminus \{ 0 \}$ is obtained by evaluating $1 + \hbar$ to some nonzero constant. Note that $q = 1$ recovers the diagram $\mathcal O_{Z_k} \vert_{\mathfrak U}$. The diagram $\mathcal A^q$ \eqref{diagramq} can thus be viewed as a ``$q$-deformation'' of $\mathcal O_{Z_k} \vert_{\mathfrak U}$ and the Abelian category of modules over this diagram can be compared to $\Qcoh (Z_k) \simeq \Qcoh (\mathcal O_{Z_k} \vert_{\mathfrak U})$ on equal terms. For example, one can find $q$-analogues of line bundles on $Z_k$ and study their moduli spaces of vector bundles or instantons. (In the context of formal deformation quantizations, these questions were studied in \cite{barmeiergasparim2}.)

One can also find a corresponding $q$-deformation $A^q$ of the endomorphism algebra $A = \End (\mathcal O_{Z_k} \oplus \mathcal O_{Z_k} (1))$ of a tilting bundle on $Z_k$ which is derived equivalent to $\mathcal A^q$ (see \S\S\ref{subsection:zktilting}--\ref{subsection:diagramtilting}).

\subsection{Deformations via the tilting bundle}
\label{subsection:zktilting}

The pullback of the tilting bundle $\mathcal O_{\mathbb P^1} \oplus \mathcal O_{\mathbb P^1} (1)$ on $\mathbb P^1$ along the projection $Z_k \toarg{\pi} \mathbb P^1$ is $\mathcal O_{Z_k} \oplus \mathcal O_{Z_k} (1)$, which is a tilting bundle on $Z_k$. We have
\begin{equation}
\label{quiverzk}
A = \End (\mathcal O_{Z_k} \oplus \mathcal O_{Z_k} (1)) \simeq \Bbbk \biggl( 
\begin{tikzpicture}[x=4em,y=1em,baseline=-.3em]
\draw[line width=1pt, fill=black] (1,0) circle(0.2ex);
\draw[line width=1pt, fill=black] (0,0) circle(0.2ex);
\node[shape=circle, scale=0.7](L) at (0,0) {};
\node[shape=circle, scale=0.7](R) at (1,0) {};
\path[->, line width=.4pt]
(L) edge[transform canvas={yshift=.6ex}] node[above=-.3ex, font=\scriptsize] {$x_0, x_1$} (R)
(L) edge[transform canvas={yshift=-.2ex}] (R)
;
\path[->, line width=.4pt, line cap=round]
(R) edge[out=-160, in=-20, transform canvas={yshift=-.6ex}] (L)
(R) edge[out=-125, in=-55, transform canvas={yshift=-.6ex}, looseness=1.2] node[below=-.4ex, font=\scriptsize] {$y_0{,}..., y_{k-1}$} (L);
\draw (0.5,-.95) node[font=\scriptsize] {$.$};
\draw (0.5,-1.15) node[font=\scriptsize] {$.$};
\draw (0.5,-1.35) node[font=\scriptsize] {$.$};
\end{tikzpicture}
\biggr) \Big/ I
\end{equation}
where $I$ is the ideal generated by
\begin{flalign}
\label{ideal}
&&\begin{aligned}
&x_1 y_{j-1} -x_0 y_j  \\
&y_{j} x_0 - y_{j-1} x_1 
\end{aligned}
&& \mathllap{1 \leq j \leq k-1}
\end{flalign}
and $A$ is the reconstruction algebra for the $\frac1k(1,1)$ singularity (cf.\ Wemyss \cite{wemyss2}). We use the usual notation $Q^{\mathrm{tilt}}$ for the quiver arising from the tilting  bundle in (\ref{quiverzk}) and label the left vertex of $Q^{\mathrm{tilt}}$ corresponding to $\mathcal O_{Z_k}$ by $0$ and the right vertex corresponding to $\mathcal O_{Z_k} (1)$ by $1$.

\subsubsection{Reduction system and Hochschild cohomology}
\label{subsubsection:zkreduction}

For $A = \Bbbk Q^{\mathrm{tilt}} / I$ as in (\ref{quiverzk}), we can give the following reduction system
\begin{equation}
\label{reductionsystem}
\begin{aligned}
R^{\mathrm{tilt}} &= \bigl\{ (x_1 y_{j-1}, x_0 y_j), (y_j x_0,y_{j-1} x_1 ) \bigr\}_{0 < j \leq k - 1} \\
\mathllap{\text{with}}\qquad S &= \{x_1 y_{j-1} , y_j x_0  \}_{0 < j \leq k - 1}
\end{aligned}
\end{equation}
and it is straightforward to check that $R^{\mathrm{tilt}}$ satisfies ($\diamond$) for $I$. Indeed the indices were chosen so that their sum is preserved by reductions and each path can thus be uniquely reduced to a path such that the indices are as large as possible towards the right.

In fact, $R^{\mathrm{tilt}}$ can be obtained from the (reduced) noncommutative Gröbner basis for $I$ with respect to the following order $\prec$. Let $p, q \in Q^{\mathrm{tilt}}_\ldot$ and let $\lvert \blank \rvert$ denote the path length.
\begin{itemize}
\item If $\lvert p \rvert < \lvert q \rvert$ set $p \prec q$.
\item If $\lvert p \rvert = \lvert q \rvert$ and $\deg (p) < \deg (q)$ set $p \prec q$, where $\deg$ is the degree defined on $\Bbbk Q$ by setting $\deg (x_i) = i$ for $i = 0, 1$ and $\deg (y_j) = j$ for $0 \leq j \leq k-1$.
\item If $\lvert p \rvert = \lvert q \rvert$ and $\deg (p) = \deg (q)$, let $\prec$ be defined as the lexicographical order which extends
$x_0 \prec y_0 \prec x_1 \prec y_1 \prec y_2 \prec \cdots \prec y_{k-2} \prec y_{k-1}$.
\end{itemize}

The reduction system $R^{\mathrm{tilt}}$ (\ref{reductionsystem}) has overlap ambiguities
\begin{align*}
S_3 &= \bigl\{ x_1 y_j x_0 \bigr\}_{0 < j < k - 1}
\end{align*}
and no higher ambiguities (i.e.\ $S_{\geq 4} = \emptyset$). Note that if $k = 2$ we also have $S_3 = \emptyset$. 

Relevant to the deformation theory are $2$-cocycles, which for $k \geq 3$ are seen to be generated as $\HH^0 (A)$-module by $\alpha_1, \dotsc, \alpha_{k-1}, \beta_0, \beta_1, \beta_2 \in \Hom(\Bbbk S, A)$, where 
\begin{flalign}
\label{equation:cocyclezk}
&&
\begin{aligned}
\alpha_i(x_1y_{j-1}) &=  \delta_{i, j} e_0 \\
\alpha_i(y_jx_0)     &= -\delta_{i, j} e_1
\end{aligned}
&&
\begin{aligned}
\beta_l (x_1y_{j-1}) &= x_0y_{j-1+l} \\
\beta_l (y_jx_0)     &= 0
\end{aligned}
&&
\end{flalign}
where $1 \leq j \leq k - 1$ and  $x_0y_{k}: = x_1y_{k-1}$. For $k = 2$ the $2$-cocycles are generated by $\alpha_1$ and $\beta_{\mathrm{symp}}$ defined as 
\begin{flalign}
\label{equation:cocyclez2}
&&
\begin{aligned}
\alpha_1 (x_1 y_0) &=  e_0 \\
\alpha_1 (y_1 x_0) &= -e_1
\end{aligned}
&&
\begin{aligned}
\beta_{\mathrm{symp}} (x_1 y_0) &= e_0 \\
\beta_{\mathrm{symp}} (y_1 x_0) &= 0.
\end{aligned}
&&
\end{flalign}
The cocycle condition for $\alpha_i, \beta_0, \beta_1, \beta_2$ and $\beta_{\mathrm{symp}}$ can easily be checked by using Remark \ref{remark:computinghh2}. (Note that for $k = 2$, $\beta_0, \beta_1, \beta_2$ are also $2$-cocycles, but they can be obtained from $\beta_{\mathrm{symp}}$ by multiplying by the paths $x_0 y_0, x_0 y_1, x_1 y_1$ viewed as elements in $\HH^0 (A) = \mathrm Z (A)$.)

Under the natural isomorphism $\HH^2 (A) \simeq \HH^2 (Z_k)$, one may obtain the following correspondence between ``algebraic'' $2$-cocycles in $\HH^2 (A)$ and ``geometric'' $2$-cocycles in $\HH^2 (Z_k)$
\begin{equation}
\label{tablecorr}
\begin{tabular}{rccc}
& \textit{commutative} & \multicolumn{2}{c}{\textit{noncommutative}} \\[.2em]
\text{algebraic} & $\alpha_j$ & \;$\beta_{\mathrm{symp}}$\; & \;$\beta_l$\; \\[.4em]
\text{geometric} & $z^{j} \tfrac{\partial}{\partial u}$ & \;$\tfrac{\partial}{\partial z} {\hair \wedge \hair} \tfrac{\partial}{\partial u}$\; & \;$z^l u \hair \tfrac{\partial}{\partial z} {\hair \wedge \hair} \tfrac{\partial}{\partial u}$\; \\[.3em]
& \footnotesize $(k \geq 2)$ & \footnotesize $(k = 2)$ & \footnotesize $(k \geq 2)$ \\
\end{tabular}
\end{equation}
where $1 \leq j \leq k - 1$ and $0 \leq l \leq 2$ (cf.\ Lemma \ref{lemma:hochschildminres}). Note that $\beta_{\mathrm{symp}}$ corresponds to the canonical holomorphic symplectic form on the open Calabi--Yau surface $\Tot \mathcal O_{\mathbb P^1} (-2) \simeq \mathrm T^* \mathbb P^1$.

\subsubsection{Commutative deformations}
\label{subsubsection:zkcommutative}

We now construct a family of deformations of $A$ which correspond to ``classical'' geometric deformations of $Z_k$. 

Let $\alpha_1, \dotsc, \alpha_{k-1}$ be as in (\ref{equation:cocyclezk}). Consider the element 
\[
\widetilde \phi =  \alpha_1 t_1 + \dotsb + \alpha_{k-1} t_{k-1} \in \Hom(\Bbbk S, A) \hatotimes (t_1, \dotsc, t_{k-1}).
\]
By Theorem \ref{theorem:higher-brackets} $\widetilde \phi$ is  a Maurer--Cartan element of the L$_\infty$ algebra $\mathfrak p (Q^{\mathrm{tilt}}, R^{\mathrm{tilt}}) \hatotimes (t_1, \dotsc, t_{k-1})$ since $\star_{\phi + \widetilde \phi}$ is associative on elements in $S_3 = \{ x_0 y_j x_1 \}_{0 < j < k-1}$ as can be easily checked.
Since each $\alpha_i$, and thus also $\widetilde \phi$, satisfies the condition \eqref{degreecondition}, the associated formal deformation admits the algebra
\[
A_{\phi + \widetilde \phi} := (A [t_1, \dotsc, t_{k-1}], \star_{\phi + \widetilde \varphi}) \simeq \Bbbk Q^{\mathrm{tilt}} [t_1, \dotsc, t_{k-1}] / I_{\phi + \widetilde \phi}
\]
as an algebraization (see Proposition \ref{proposition:Groebner}), where $I_{\phi + \widetilde \phi}$ is the ideal generated by
\begin{flalign*}
&&\begin{aligned}
x_1 y_{j-1} - x_0 y_j   - e_0 t_j \\
y_j x_0  -  y_{j-1} x_1 + e_1 t_j
\end{aligned}
&&
\mathllap{1 \leq j \leq k-1.}
\end{flalign*}
Evaluating this algebraization to $t_i \tikzmapsto \lambda_i$ for some $\lambda = (\lambda_1, \dotsc, \lambda_{k-1}) \in \Bbbk^{k-1}$ we obtain an actual deformation $A_\lambda := A_{\phi + \widetilde \varphi}\rvert_{t_i = \lambda_i}$ of $A$. 

\begin{proposition}
\label{prop:zkcommutative}
$A_\lambda$ is Morita equivalent to its center $\mathrm Z (A_\lambda ) \simeq e_0 A_\lambda  e_0 \simeq e_1 A_\lambda  e_1$ precisely when $\lambda = (\lambda_1, \dotsc, \lambda_{k-1}) \neq 0$. In this case we have algebra isomorphisms
$$
\mathrm Z (A_\lambda ) \simeq e_i A_\lambda  e_i \simeq \Bbbk [z_0, \dotsc, z_k] \big/ \bigl( \mathrm{rank} \bigl(
\begin{smallmatrix}
z_0 & z_1 + \lambda_1 & z_2 + \lambda_2 & \cdots & z_{k-1} + \lambda_{k-1} \\
z_1 & z_2             & z_3             & \cdots & z_k
\end{smallmatrix}
\bigr) \leq 1 \bigr).
$$
\end{proposition}

\begin{proof} 
It follows from \cite[Prop.~1.9]{buchweitz2} that $A_\lambda$ is Morita equivalent to $e_0 A_\lambda e_0$ if and only if the restricted multiplication map
\begin{align*}
\begin{tikzpicture}[baseline=-2.6pt,description/.style={fill=white,inner sep=1.75pt}]
\matrix (m) [matrix of math nodes, row sep=0em, text height=1.5ex, column sep=1.5em, text depth=0.25ex, ampersand replacement=\&, column 3/.style={anchor=base west}, column 2/.style={anchor=base west}]
{
\mu_{e_0} \colon \&[-2em] A_\lambda  e_0 \otimes_{e_0 A_\lambda  e_0} e_0 A_\lambda  \& A_\lambda  \\
\& \phantom{A_\lambda  e_0} \mathllap{a} \otimes_{e_0 A_\lambda  e_0} b \& a b \\
};
\path[->,line width=.4pt]
(m-1-2) edge (m-1-3)
;
\path[|->,line width=.4pt]
(m-2-2) edge (m-2-3)
;
\end{tikzpicture}
\end{align*}
is surjective. If there exists $1 \leq j \leq k-1$ such that $\lambda_j \neq 0$, then we have 
$$
\mu_{e_0} \bigl( e_0 \otimes_{e_0 A_\lambda  e_0} e_0 + \tfrac{1}{\lambda_j} y_{j-1} \otimes_{e_0 A_\lambda  e_0} x_1 - \tfrac{1}{\lambda_j} y_j \otimes_{e_0 A_\lambda  e_0} x_0 \bigr) = e_0 + e_1 = 1.
$$
Thus $\mu_{e_0}$ is surjective. If $(\lambda_1, \dotsc, \lambda_{k-1}) = 0$ then 
$e_1 \notin \im (\mu_{e_0})$. This yields that $A_\lambda $ is Morita equivalent to $e_0 A_\lambda  e_0$ if and only if $(\lambda_1, \dotsc, \lambda_{k-1}) \neq 0$. We obtain that their centers are isomorphic as algebras, namely $\mathrm Z (A_\lambda ) \cong e_0A_\lambda e_0$.

It is not difficult to show that there is an algebra isomorphism 
$$
 f_0 \colon \Bbbk [z_0, \dotsc, z_k] \Big/ \Big(\mathrm{rank} \bigl( \begin{smallmatrix}
z_0 & z_1 + \lambda_1 & \dotsc &  z_{k-1} + \lambda_{k-1}\\
z_1 & z_2 & \dotsc & z_k
\end{smallmatrix} \bigr) \leq 1 \Big)\tosim e_0A_\lambda e_0.$$
determined by $f_0 (z_0) = -x_0 y_0, \ f_0(z_k) = -x_1 y_{k-1}, \ f_0 (z_j) = -x_0 y_j - \lambda_j e_0$ for $1 \leq j \leq k-1$.
Similarly, we have an algebra isomorphism 
$$f_1 \colon  \Bbbk [z_0, \dotsc, z_k] \Big/ \Big(\mathrm{rank} \bigl( \begin{smallmatrix}
z_0 & z_1 + \lambda_1 & \dotsc &  z_{k-1} + \lambda_{k-1}\\
z_1 & z_2 & \dotsc & z_k
\end{smallmatrix} \bigr) \leq 1 \Big)\tosim  e_1A_\lambda e_1$$
determined by $ f_1(z_0)= - y_0x_0$ and $f_1(z_j)= - y_{j-1} x_1$ for $1\leq j \leq k-1$. 
\end{proof}

Geometrically Proposition \ref{prop:zkcommutative} reflects the fact that the nontrivial commutative deformations of $Z_k$ induced by $1$-cocycles in $\H^1 (\mathcal T_{Z_k})$ are smooth affine varieties (see \cite[Thm.~6.18]{barmeiergasparim1}) with coordinate rings $\mathrm Z (A_\lambda)$. From a singularity point of view, these affine varieties are obtained from smoothening the $\frac{1}k (1,1)$ singularity obtained by contracting the exceptional $\mathbb P^1 \subset Z_k$ with self-intersection $-k$ (cf.\ \S\ref{subsection:xk}).

\subsubsection{Noncommutative deformations}
\label{subsubsection:zknoncommutative}

We now also construct a family of deformations of $A$ corresponding to quantizations of Poisson structures.

Let $k \geq 3$ and let $\beta_0, \beta_1, \beta_2$ be as in (\ref{equation:cocyclezk}). Using Theorem \ref{theorem:higher-brackets}, one checks that the $2$-cocycle $\beta_0 t_0 + \beta_1 t_1 + \beta_2 t_2$ is itself not a Maurer--Cartan element of $\mathfrak p (Q^{\mathrm{tilt}}, R^{\mathrm{tilt}}) \hatotimes (t_0, t_1, t_2)$. However, we may consider the following element $\widetilde \phi \in \Hom(\Bbbk S, A) \hatotimes (t_0, t_1, t_2)$, where $\widetilde \phi(y_{j}x_0) = 0$ and $\widetilde \phi(x_1y_{j-1})$ equals 
\begin{align*}
-x_0y_j+   x_0y_{j-1}t_0 + \sum_{i=j}^{k-2} x_0 y_{i} t_2^{i-j}(1+t_1+t_0t_2) + x_0y_{k-1}t_2^{k-j-1}(1+t_1) + x_1y_{k-1}t_2^{k-j}
\end{align*}
for each $1 \leq  j \leq  k-1$. For instance, $\widetilde \phi(x_1y_{k-2}) = x_0y_{k-2} t_0 + x_0y_{k-1} t_1 + x_1y_{k-1}t_2$. That is, $\widetilde \phi$ ``corrects'' the $2$-cocycle $\beta_0 t_0 + \beta_1 t_1 + \beta_2 t_2$ to a Maurer--Cartan element of $\mathfrak p (Q^{\mathrm{tilt}}, R^{\mathrm{tilt}}) \hatotimes (t_0, t_1, t_2)$ by adding higher-order terms.

\begin{proposition}
\label{prop:zknoncommutative}
\begin{enumerate}
\item \label{zknc1} $\widetilde \phi$ is a Maurer--Cartan element of $\mathfrak p (Q^{\mathrm{tilt}}, R^{\mathrm{tilt}}) \hatotimes (t_0, t_1, t_2)$.
\item \label{zknc2} The formal deformation associated to $\widetilde \phi:=\beta_0 t_0 + \beta_1 t_1$ admits an algebraization 
\[
A_{\phi + \widetilde \phi} = (A[t_0, t_1], \star_{\phi+ \widetilde \varphi}) = \Bbbk Q^{\mathrm{tilt}} [t_0, t_1]/ I_{\phi + \widetilde \phi}
\]
where $I_{\phi + \widetilde \phi}$ is the two-sided ideal generated by 
\begin{gather*}
x_1 y_{j-1} - x_0 y_{j-1} t_0 - x_0 y_{j}(1 + t_1)\\
y_j x_0 - y_{j-1}x_1 
\end{gather*}
for $1 \leq j \leq k-1$. Evaluating the algebraization at $\mu = (\mu_0, \mu_1) \in \Bbbk^2$ we obtain actual deformations of $A$ corresponding to quantizations of Poisson structures on $Z_k$.
\end{enumerate}
\end{proposition}

\begin{proof}
By Theorem \ref{theorem:higher-brackets} it suffices to verify that $\star:=\star_{\phi + \widetilde \phi}$ is associative on elements in
$
S_3 = \{ x_1 y_j x_0 \}_{0 < j < k-1}.
$
This follows since both $(x_1 \star y_j) \star x_0$ and $x_1 \star (y_j \star x_0)$ equal 
$$
x_0y_{j-1}x_1t_0 + \sum_{i=j}^{k-2} x_0 y_{i} x_1t_2^{i-j}(1+t_1+t_0t_2) + x_0y_{k-1}x_1 t_2^{k-j-1}(1+t_1) + x_1y_{k-1}x_1 t_2^{k-j}.
$$

Now \ref{zknc2} follows from Proposition \ref{proposition:Groebner}, noting that $g = \beta_0 t_0 + \beta_1 t_1$ satisfies the degree condition ($\prec$), since $x_0 y_{j-1}, x_0 y_j \prec x_1 y_{j-1}$.
\end{proof}

\subsubsection{The Calabi--Yau case}
\label{subsubsection:z2}

For $k = 2$, the $\frac12(1,1)$ singularity is just the $\mathrm A_1$ singularity whose minimal resolution $Z_k \simeq \Tot \mathcal O_{\mathbb P^1} (-2) \simeq \mathrm T^* \mathbb P^1$ is an open Calabi--Yau surface and $A$ is the preprojective algebra of type $\widetilde{\mathrm A}_1$.

Let $R^{\mathrm{tilt}}$ be the reduction system given in \eqref{reductionsystem}. Since $S_3 = \emptyset$ it follows that any element $\widetilde \varphi \in \Hom(\Bbbk S, A) \hatotimes \mathfrak m$ is a Maurer--Cartan element of $\mathfrak p (Q^{\mathrm{tilt}}, R^{\mathrm{tilt}}) \hatotimes \mathfrak m$. Consider the Maurer--Cartan element $\widetilde \phi = \alpha_1 t_1 + \beta_{\mathrm{symp}} t_2$ of $\mathfrak p (Q^{\mathrm{tilt}}, R^{\mathrm{tilt}}) \hatotimes (t_1, t_2)$ (cf.\ \eqref{equation:cocyclez2}). The corresponding formal deformation $(A \llrr{t_1, t_2}, \star_{\phi+ \widetilde \varphi})$ admits an algebraization 
\begin{align}\label{Z2-2}
(A [t_1, t_2], \star_{\phi+ \widetilde \varphi})\simeq \Bbbk Q^{\mathrm{tilt}} [t_1, t_2] / I_{\phi+\widetilde \phi}
\end{align}
where $I_{\phi+\widetilde \phi}$ is the two-sided ideal generated by
\begin{gather*}
x_1 y_0 - x_0 y_1 - e_0 t_1 - e_0 t_2 \\
y_1 x_0 - y_0 x_1 + e_1 t_1.
\end{gather*}

Denote by $A_{\lambda, \mu}$ the algebra evaluating the algebraization \eqref{Z2-2} at $t_1 = \lambda$ and $t_2 = \mu$ for some $\lambda, \mu \in \Bbbk$.

\begin{proposition}
\label{proposition:z2}
\begin{enumerate}
\item \label{Z2-3} The subalgebras $e_0 A_{\lambda, \mu} e_0$ and $e_1 A_{\lambda, \mu} e_1$ are commutative if and only if $\mu = 0$ in which case $A_{\lambda, 0}$ for $\lambda \neq 0$ is Morita equivalent to a commutative deformation of the $\mathrm A_1$ surface singularity.
\item \label{Z2-4} If $\mu \neq 0$ then we have that
\[
e_0 A_{\lambda, \mu} e_0 \simeq \U (\mathfrak{sl}_2) \bigm/ \bigl( C - \textstyle\frac{\lambda^2 - \mu^2}{2 \mu^2} \bigr) \quad\text{and}\quad e_1 A_{0,1} e_1 \simeq \U (\mathfrak{sl}_2) \bigm/ \bigl( C - \frac{\lambda^2+ 2 \lambda \mu}{2\mu^2} \bigr)
\]
where $\U (\mathfrak{sl}_2)$ is the universal enveloping algebra of $\mathfrak{sl}_2 = \langle H, X, Y \rangle$ and $C = XY + YX + \frac12 H^2$ is the Casimir element.
\end{enumerate}
\end{proposition}

\begin{proof}
Assertion \ref{Z2-3} follows from \cite[Thm.~0.4]{crawleyboeveyholland} and also from Proposition \ref{prop:zkcommutative} for $k = 2$. To see \ref{Z2-4} one verifies that there is an algebra isomorphism 
$f_0 \colon \U(\mathfrak{sl}_2)/(C - \frac{\lambda^2 - \mu^2}{2 \mu^2}) \tikzto e_0 A_{0, 1} e_0$ given by 
\[
f_0(H)=\frac{2}{\mu} x_0y_1 + \frac{\lambda+\mu}{\mu} e_0, \quad f_0(X)=-\frac{1}{\mu} x_0y_0 \quad \text{and} \quad f_0(Y)=\frac{1}{\mu} x_1y_1.
\]
Similarly, we have an algebra isomorphism $f_1 \colon \U (\mathfrak{sl}_2) / (C - \frac{\lambda^2+ 2 \lambda \mu}{2\mu^2}) \tikzto e_1 A_{0, 1} e_1$ given by
\begin{flalign*}
&& f_1 (H) = \frac{2}{\mu} y_0x_1 - \frac{\lambda}{\mu} e_1, \quad f_1(X) = -\frac{1}{\mu} y_0x_0 \quad \text{and} \quad f_1 (Y) = \frac{1}{\mu} y_1x_1. && \qedhere
\end{flalign*}
\end{proof}

\begin{remark}\label{remark:singularityinvariant}
Since $k = 2$, note that the algebra $A_{\lambda, \mu}$ is the deformed preprojective algebra of type $\widetilde{\mathrm A}_{1}$ corresponding to $(\lambda + \mu, -\lambda)$ with $e_0$ being the special vertex.  It follows from Crawley-Boevey--Holland \cite[Thm.~0.4~(4)]{crawleyboeveyholland} that $e_0 A_{\lambda, \mu} e_0$ has {\it infinite} global dimension if and only $\lambda = 0$, where $\lambda$ corresponds to the ``commutative'' direction in $\HH^2 (A)$. In this case we have that $e_0 A_{0, \mu} e_0 \simeq \U (\mathfrak{sl}_2) / (C + \frac{1}{2})$ and it follows from Crawford \cite[Cor.~1.2.6]{crawford} that there is a triangulated equivalence of singularity categories $\D_{\mathrm{sg}} (e_0 A_{0,\mu} e_0) \simeq \D_{\mathrm{sg}} (e_0 A \hair e_0)$. The latter is the singularity category of the $\mathrm A_1$ singularity, which stays unchanged under ``purely noncommutative'' deformations induced by the holomorphic symplectic structure on $Z_2$, but becomes trivial whenever one also deforms in a commutative direction (i.e.\ $\lambda \neq 0$).
\end{remark}

\subsubsection{Variety of simultaneous deformations}
\label{subsubsection:diagramvariety}

Recall from \S\ref{subsubsection:zkreduction} that the reduction system $R^{\mathrm{tilt}}$ is obtained from the noncommutative Gröbner basis for $I$ with respect to the order $\prec$. Using Theorem \ref{theorem:variety} we obtain an algebraic variety $V_{\prec} \subset \Hom (\Bbbk S, A) \simeq \mathbb A^N$ of actual deformations of $A$ and a natural groupoid $G_\prec \tikzrightrightarrows V_{\prec}$ whose orbits correspond to the isomorphism classes of actual deformations.

If $k = 2$, there are no obstructions and the variety $V_\prec$ of simultaneous deformations is simply the space $\Hom (\Bbbk S, A)_\prec \simeq \mathbb A^6$. We note that the $3$-dimensional linear subspace spanned by $\alpha_1$, $\beta_{\mathrm{symp}}$ and $\beta_1$ give rise to the so-called {\it deformed quantum} preprojective algebras of type $\widetilde{\mathrm A}_1$ studied by Kalck \cite{kalck} and Crawford \cite{crawford}, where $\alpha_1$ and $\beta_{\mathrm{symp}}$ correspond to the ``deformed'' part and $\beta_1$ to the ``quantum'' part as $\beta_1$ is the term giving rise to a ``$q$-deformation'' of the preprojective relations.

Now let $k \geq 3$. In this case, there are obstructions to simultaneous deformations, so it is convenient to work with a smaller set of $2$-cocycles in $\Hom (\Bbbk S, A)_\prec$. Let $H \subset \Hom (\Bbbk S, A)_\prec$ be the $\Bbbk$-linear subspace spanned by the $2$-cocycles $\alpha_1, \dotsc, \alpha_{k-1}, \beta_0, \beta_1$ \eqref{equation:cocyclezk} so that $H \simeq \mathbb A^{k+1}$. (Note that $\beta_2 \not\in \Hom (\Bbbk S, A)_\prec$.) Now consider the element  
\[
\widetilde \phi = \mu_0 \beta_0 + \mu_1 \beta_1 + \sum_{j=1}^{k-1} \lambda_j \alpha_j \in H \subset \Hom(\Bbbk S, A)_\prec
\]
with $\mu_0, \mu_1, \lambda_1, \dotsc, \lambda_{k-1} \in \Bbbk $. Then $\widetilde \phi$ is a Maurer--Cartan element of $\mathfrak p (Q^{\mathrm{tilt}}, R^{\mathrm{tilt}})$ if and only if for each $1 \leq j \leq k-2$ we have 
$(x_1 \star y_j) \star x_0 = x_1 \star (y_j \star x_0)$ 
where $\star = \star_{\phi + \widetilde \phi}$.
This is equivalent to 
\begin{flalign}
\label{align:euqationvariety}
&& \mu_0 \lambda_j + \mu_1 \lambda_{j+1} = 0  &&  \llap{for each $ 1 \leq j \leq k-2$.}
\end{flalign}
Thus we obtain the following result. 

\begin{proposition}
\label{proposition:varietyzk}
For $k \geq 3$, the equations \eqref{align:euqationvariety} cut out an affine variety $V = H \cap V_\prec \subset \mathbb A^{k+1}$ of simultaneous deformations of $A$.
\end{proposition}

For example, for $k = 3$ the variety $V$ in Proposition \ref{proposition:varietyzk} is isomorphic to an ordinary double point threefold singularity (conifold) given by $x y - z w = 0$ in $\mathbb A^4$.

In \S\ref{subsubsection-simultaneousdeformations}, this variety will also be viewed as a variety of actual simultaneous deformations of the $\frac{1}k (1,1)$ singularity.

\subsection{Diagram algebra {\it vs} tilting bundle}
\label{subsection:diagramtilting}

Recall from Proposition \ref{proposition:derived} that on the level of formal deformations, the two approaches to deformations of $\Qcoh (X)$ --- via deformations of the diagram algebra $\mathcal O_X \vert_{\mathfrak U}!$ or via deformations of the endomorphism algebra $\End \mathcal E$ of a tilting bundle --- are derived equivalent.

This also holds for actual deformations, as can be checked directly for the actual deformations corresponding to commutative (\S\ref{subsubsection:zkcommutativediagram} and \S\ref{subsubsection:zkcommutative}) and noncommutative deformations (\S\ref{subsubsection:diagramnoncommutative} and \S\ref{subsubsection:zknoncommutative}).

Although the derived equivalence between $\Qcoh (Z_k)$ and $\Mod (A)$ lifts to their formal and actual deformations, it may happen that $A$ admits an algebraization that does not appear to give rise to an algebraization of $\mathcal O_{Z_k} \vert_{\mathfrak U}$.\footnote{A similar phenomenon can also be observed for $\mathbb P^2$ when comparing deformations of the diagram $\mathcal O_{\mathbb P^2} \vert_{\mathfrak U}$, where $\mathfrak U$ is the closure of the standard affine open cover $\{ U_0, U_1, U_2 \}$ under intersections, to deformations of the endomorphism algebra of the tilting bundle $\mathcal E = \mathcal O_{\mathbb P^2} \oplus \mathcal O_{\mathbb P^2} (1) \oplus \mathcal O_{\mathbb P^2} (2)$. Deformations of the $15$-dimensional algebra $\End \mathcal E$ are unobstructed, so one can obtain a versal family of deformations over $(\mathbb A^{10}, 0)$, where $\dim \HH^2 (\mathbb P^2) = 10$, corresponding geometrically to quantizations of Poisson structures on $\mathbb P^2$. However, only some of these Poisson structures give rise to algebraizable deformations of the diagram $\mathcal O_{\mathbb P^2} \vert_{\mathfrak U}$.}

In particular, consider the element 
$\widetilde \Phi \in \Hom(\Bbbk S, \Bbbk Q^{\mathrm{diag}} / J) \hatotimes (t_0', t_1', t_1, \dotsc, t_{k-1})$ given by 
\begin{flalign*}
&& \widetilde \Phi(wx) &= \alpha(x) w & \widetilde \Phi(wy)      &= \beta(y) w     & \widetilde \Phi(uf) &= \sum_{j=1}^{k-1} f y^{j} t_{j} && \\[-1em]
&& \widetilde \Phi(uz) &= \alpha(z) u & \widetilde \Phi(v \zeta) &= \beta(\zeta) v &&&&
\end{flalign*}
where 
$\alpha(x) = t_0' +  x t_1'$ and $\beta(y) = \sum_{i=1}^{\infty} y (-yt_0' - t_1')^i$. That is, $\widetilde \Phi$ corresponds to arbitrary commutative and certain noncommutative deformations --- indeed exactly those corresponding to the $2$-cocycles $\alpha_1, \dotsc, \alpha_{k-1}, \beta_0, \beta_1$ in \S\ref{subsubsection:diagramvariety}. By Theorem \ref{theorem:higher-brackets} it follows that $\widetilde \Phi$ is a Maurer--Cartan element if and only if 
\begin{flalign}\label{align:euqationvariety1}
&& t_0' t_j + t_1' t_{j+1} = 0  &&  \llap{for each $ 1 \leq j \leq k-2$.}
\end{flalign}
We thus obtain a {\it formal} deformation $({\mathcal O_{Z_k} \vert_{\mathfrak U}!} \hatotimes B, \star_{\phi + \widetilde \phi})$ of the diagram algebra $\mathcal O_{Z_k} \vert_{\mathfrak U}!$ over the complete local Noetherian $\Bbbk$-algebra
\[
B = \Bbbk \llrr{t_0', t_1', t_1, \dotsc, t_{k-1}} / (t_0' t_j + t_1' t_{j+1})^{\widehat{\;\;}}_{1 \leq j \leq k-2}.
\]
Although we saw in \S\ref{subsection:zkdiagram} that all purely commutative deformations of $\mathcal O_{Z_k} \vert_{\mathfrak U}!$ and certain purely noncommutative deformations admit algebraizations, the formal simultaneous deformation over $B$ does not appear to admit an obvious algebraization. However, comparing \eqref{align:euqationvariety} and \eqref{align:euqationvariety1}, we note that $B$ is exactly the formal completion of the coordinate ring of the variety $V$ of actual deformations of $A$ given in Proposition \ref{proposition:varietyzk}, reflecting the fact that $\mathcal O_{Z_k} \vert_{\mathfrak U}!$ and $A = \End (\mathcal O_{Z_k} \oplus \mathcal O_{Z_k} (1))$ have the same {\it formal} deformation theory.

\section{Deformations of singularities via their noncommutative resolutions}
\label{section:singularities}

In this section we apply the techniques developed thus far to study deformations of singularities via their noncommutative resolutions. Let $X = \Spec C$ be an affine variety with an isolated singularity. The commutative deformation theory of the singularity is well understood: deformations are parametrized by the second Harrison cohomology group $\Har^2 (C)$, which is finite-dimensional when the singularities are isolated, with obstructions in $\Har^3 (C)$ (see e.g.\ Stevens \cite{stevens2}). For many classes of singularities versal deformation spaces have been constructed \cite{altmann2,stevens,stevens2}. These versal deformation spaces are varieties of ``actual'' deformations which completely describe the commutative deformation theory of the singularity. In particular, the formal completion of the versal family at the point corresponding to the original algebra captures the formal deformation theory of $C$.

However, $\Har^2 (C)$ is usually strictly contained in $\HH^2 (C)$ and the latter is often infinite-dimensional even if $C$ has an isolated singularity (see \S\ref{subsection:xk} below for an example). As $\HH^2 (C)$ parametrizes associative deformations of $C$, an unobstructed $2$-cocycle will parametrize a {\it formal} noncommutative deformation of $C$, but in general formal noncommutative deformations of $C$ will not necessarily admit an algebraization, so that one cannot expect a versal deformation space for the full associative deformation theory of $C$ to exist.

However, using the notion of an admissible order for a suitable quiver with relations, one can construct varieties $V_\prec$ of actual deformations of $C$ which contain at least some (often interesting) noncommutative deformations. In \S\ref{subsection:noncommutativeresolutions} we show how to use noncommutative resolutions of the singularity to obtain such varieties and in \S\ref{subsection:xk} we illustrate this for the cyclic $\frac{1}k (1,1)$ surface quotient singularities, which can be linked to the geometry of the deformation theory of $\Qcoh (Z_k)$ described in \S\ref{section:totalspaces}.

\subsection{Deformations of noncommutative resolutions}
\label{subsection:noncommutativeresolutions}

Let $C$ be a (singular) Noetherian commutative $\Bbbk$-algebra. A {\it noncommutative resolution} of $C$ is a $\Bbbk$-algebra of the form $A = \End_C (C \oplus M)$, where $M$ is a finitely generated $C$-module, such that $A$ has finite global dimension (see e.g.\ \cite{vandenbergh0, daoiyamatakahashivial}). Let $e \in A$ be the idempotent of $A$ corresponding to the direct summand $C$ of $C \oplus M$ so that $C = eAe$. Writing $A$ as $\Bbbk Q / I$, with $e$ corresponding to a vertex of $Q$, we have the following general result.

\begin{proposition}
\label{theorem:singularitypqr}
Let $A = \End_C (C \oplus M) \simeq \Bbbk Q / I$ be a noncommutative resolution of $C$. Let $R$ be a reduction system satisfying \textup{($\diamond$)} for $I$. Then $\mathfrak p (Q, R)$ controls the deformation theory of $A$. 

Moreover, for any Maurer--Cartan element $\widetilde \phi$ of $\mathfrak p (Q, R) \hatotimes \mathfrak m$, the deformation $\widehat A_{\phi+\widetilde \phi}$ of $A$ induces a deformation $e \widehat A_{\phi + \widetilde \phi} e$ of $C$. 

For any degree condition \eqref{degreecondition}, the variety $V_\prec$ of actual deformations of $A$ gives rise to a family of actual deformations of $C$. 
\end{proposition}

\begin{proof}
The first assertion follows from Theorem \ref{theorem:BW1}. For the second assertion, note that the restriction map $\mathfrak h (A) \tikzto \mathfrak h (eAe)$ is a morphism of DG Lie algebras. Composing this with the L$_\infty$ quasi-isomorphism in Theorem \ref{theorem:BW1}, we obtain an L$_\infty$ morphism $\mathfrak p (Q, R) \tikzto \mathfrak h (eAe)$ which thus induces a map between the deformation functors. For any $A_\lambda \in V_\prec$, $e A_\lambda e$ is an actual deformation of $C = e A e$.
\end{proof}

Note that noncommutative resolutions of a commutative singularity can sometimes be obtained from geometric resolutions of the singularity, as is the case for cyclic quotient surface singularities (see Wemyss \cite{wemyss2} and \S\ref{subsection:xk} below for a particular case).

Finding ``actual'' noncommutative deformations of singularities is in general difficult. By using Kontsevich's formality morphism Filip \cite{filip} showed that Poisson structures on singular affine toric varieties can always be formally quantized. From this formal statement, however, it is unclear how to construct any explicit algebraizable quantizations. Theorem \ref{theorem:singularitypqr} shows that we can use the degree condition \eqref{degreecondition} to construct actual noncommutative deformations for example via noncommutative resolutions of the singularity, giving a whole variety $V_\prec$ of actual deformations.

\subsection{Deformations and singularity categories}

Recall that the singularity category $\D_{\mathrm{sg}} (A)$ of a Noetherian (not necessarily commutative) algebra $A$ is defined as the Verdier quotient of the bounded derived category of finitely generated $A$-modules by the full subcategory of perfect complexes. This notion was first introduced by Buchweitz \cite{buchweitz1} and then rediscovered by Orlov \cite{orlov} in the context of Landau--Ginzburg models in homological mirror symmetry.

One may now ask whether the singularity category changes under deformations of the singularity. In particular, about the variety $V_\prec$ obtained from actual deformations of a noncommutative resolution we can ask the following question.

\begin{question}
\label{question}
For which points $v \in V_{\prec}$ does $e A_v e$ have the same singularity category as $C = e A e$?
\end{question} 

The singularity category of a singular algebra is always nontrivial. In general we expect the singularity category of $e A_v e$ to become  ``smaller'' for {\it commutative} deformations. For example, if the induced deformation of the singularity is smooth, the singularity category is trivial. On the other hand, Remark \ref{remark:singularityinvariant} showed that for a certain purely noncommutative deformation of the $\mathrm A_1$ singularity, the singularity category remained invariant. Indeed, it seems natural to conjecture that ``purely noncommutative'' deformations of commutative isolated singularities leave the singularity category invariant.

\subsubsection{A deformation-theoretical perspective}

In order to determine whether two algebras have equivalent singularity categories, one generally needs a larger set of tools and techniques than those introduced in this article. (See for example Kalck--Karmazyn \cite{kalckkarmazyn} for singular equivalences between cyclic quotient surface singularities and certain finite-dimensional algebras.) However, a general deformation-theoretical framework for answering the above question is the following. We have an L$_\infty$ morphism 
\begin{equation}
\label{morphism}
\mathfrak p(Q, R) \tikzto \mathfrak h (eAe) \tikzhookrightarrow \mathfrak h_{\mathrm{sg}} (eAe) 
\end{equation}
where the first arrow is the L$_\infty$ morphism given in the proof of Proposition \ref{theorem:singularitypqr} and $\mathfrak h_{\mathrm{sg}} (eAe) = (\mathrm C_{\rm sg}^{\hdot + 1}(eAe, eAe), d, [\blank {,} \blank])$ is a DG Lie algebra whose underlying cochain complex is the (shifted) {\it singular Hochschild cochain complex} $\mathrm C_{\mathrm{sg}}^\hdot (eAe, eAe)$ which computes the singular Hochschild cohomology $\HH_{\rm sg}^\hdot (eAe)$ (see \cite{wang}). The DG Lie algebra $\mathfrak h_{\mathrm{sg}} (eAe)$ admits $\mathfrak h (eAe)$ as a DG Lie subalgebra and Keller's conjecture in \cite{keller} implies in particular that the deformation theory of the DG singularity category of $eAe$ is controlled by $\mathfrak h_{\mathrm{sg}} (eAe)$.

The L$_\infty$ morphism \eqref{morphism} induces a map between the Maurer--Cartan spaces. Let $\widetilde \phi$ be a Maurer--Cartan element of $\mathfrak p(Q, R)$ and let $A_{\phi + \widetilde \phi}$ be the corresponding (actual) deformation of $A$. If the image of $\widetilde \phi$ is gauge equivalent to the zero Maurer--Cartan element of $\mathfrak h_{\mathrm{sg}} (eAe)$, then by the general philosophy of deformation theory, we expect that the singularity category $\D_{\rm sg} (e A_{\phi + \widetilde \phi} e)$ does not change. In other words, one expects that a deformation has the same singularity category when the corresponding $2$-cocycle lies in the kernel of the map $\HH^2 (A) \tikzto \HH_{\rm sg}^2 (eAe)$ induced by \eqref{morphism}.

Let us point to some evidence for this. If $eAe$ is an affine hypersurface with an isolated singularity then $\HH_{\rm sg}^2 (eAe)$ is isomorphic to the Harrison cohomology $\Har^2 (eAe)$ (see Lemma \ref{lemma:hh2hypersurface} and \cite[\S4]{keller}) so that the kernel of the map $\HH^2(eAe) \tikzto \HH^2_{\rm sg}(eAe)$ coincides with the direct summand $\mathcal N$ in Lemma \ref{lemma:hh2hypersurface} which correspond to purely noncommutative deformations. In line with the above reasoning, this gives a deformation-theoretic interpretation for Crawford's singular equivalence in Remark \ref{remark:singularityinvariant} since the $2$-cocycle corresponding to the deformation $e_0A_{0, \mu} e_0 \simeq \U (\mathfrak{sl}_2) / (C + \frac{1}{2})$ lies in $\mathcal N$ (see Remark \ref{remark:hypersurfaces}). Note that the image of the $2$-cocyle corresponding to the deformation $e_0 A_{\lambda, \mu} e_0$ for $\lambda \neq 0$ is nonzero along the map $\HH^2(eAe) \tikzto \HH^2_{\rm sg}(eAe)$ and the singularity category $\D_{\rm sg}(e_0 A_{\lambda, \mu} e_0)$ does indeed change --- it is trivial since $e_0 A_{\lambda, \mu} e_0$ is of finite global dimension.

\subsection{The $\frac{1}k (1,1)$ surface singularities}
\label{subsection:xk}

We now apply the above scheme of studying noncommutative deformations of singularities via their noncommutative resolutions to the $\frac{1}k (1,1)$ singularity, which also allows us to apply the results obtained in \S\ref{section:totalspaces}. In the rest of this section we work over $\Bbbk = \mathbb C$.

Let $X_k$ be the $\frac{1}k (1,1)$ singularity
\[
X_k = \mathbb C^2 / \Gamma = \Spec C
\]
where
\[
C = \mathbb C [z_0, \dotsc, z_k] / (z_i z_{j+1} - z_{i+1} z_j)_{0 \leq i < j < k}
\]
and $\Gamma < \GL_2 (\mathbb C)$ is a cyclic group of order $k$ with the generator acting on $\mathbb C^2$ by
\[
\begin{pmatrix}
\omega & 0 \\
0 & \omega
\end{pmatrix}
\]
for $\omega$ some primitive $k$th root of unity. (For $k = 2$ we have that $\Gamma < \SL_2 (\mathbb C)$ and the $\frac12 (1,1)$ singularity is the ``usual'' $\mathrm A_1$ surface singularity.)

Note that the coordinate ring $C$ of $X_k$ is the ring of global functions on $Z_k = \Tot \mathcal O_{\mathbb P^1} (-k)$ and $Z_k$ is the minimal resolution of $X_k$. Recall from \S\ref{section:totalspaces} that $Z_k$ admits a tilting bundle $\mathcal O_{Z_k} \oplus \mathcal O_{Z_k} (1)$ whose endomorphism algebra $A$ is given in \eqref{quiverzk}. Since $Z_k$ is smooth, $A$ is of finite global dimension and in fact \cite{wemyss2} $A$ is a noncommutative resolution of $X_k$ (cf.\ \S\ref{subsection:noncommutativeresolutions}). Indeed, letting $e_0$ denote the idempotent corresponding to the direct summand $\mathcal O_{Z_k}$, we have $e_0 A e_0 = C$.

Writing $A = \mathbb C Q^{\mathrm{tilt}} / I$ and letting $R^{\mathrm{tilt}}$ be the reduction system \eqref{reductionsystem}, we saw that $\mathfrak p (Q^{\mathrm{tilt}}, R^{\mathrm{tilt}})$ controls the deformation theory of $\Qcoh (Z_k)$. As in \eqref{morphism} we have an L$_\infty$ morphism
\[
\mathfrak p (Q^{\mathrm{tilt}}, R^{\mathrm{tilt}}) \tikzto \mathfrak h (C).
\]
Deformations of $\Qcoh (Z_k)$ described in \S\ref{section:totalspaces} thus induce deformations of $X_k = \Spec C$.

\subsubsection{Commutative deformations}

Let us first recall the commutative deformation theory of the singularity $X_k = \Spec C$. Note that the relations of $C$ can be written in matrix form as follows:
\begin{align}
C &\simeq \mathbb C [z_0, \dotsc, z_k] \Bigm/ \Bigl( \mathrm{rank} \begin{psmallmatrix}
z_0 & z_1 & \cdots & z_{k-2} & z_{k-1}\\
z_1 & z_2 & \cdots & z_{k-1} & z_k
\end{psmallmatrix} \leq 1 \Bigr) \label{presentation1} \\
&\simeq \mathbb C [z_0, \dotsc, z_k] \Bigm/ \Bigl( \mathrm{rank} \begin{psmallmatrix}
z_0 & z_1 & \cdots & z_{k-3} & z_{k-2} \\
z_1 & z_2 & \cdots & z_{k-2} & z_{k-1} \\
z_2 & z_3 & \cdots & z_{k-1} & z_k
\end{psmallmatrix} \leq 1 \Bigr). \label{presentation2}
\end{align}
The commutative deformations of $C$ are parametrized by the Harrison cohomology $\Har^2 (C) \subset \HH^2 (C)$. We have a decomposition
\[
\Har^2 (C) \simeq \mathbb C^{k-1} \oplus \mathbb C^{k-3}
\]
and for all $k \geq 2$ there exists a versal family with a component $\mathbb A^{k-1}$ called {\it Artin component} corresponding to the first summand of $\Har^2 (C)$ (see for example \cite{altmann,pinkham,riemenschneider,stevens}). For $k = 4$ the versal family also contains a second irreducible component $\mathbb A^1$ and for $k > 4$ an embedded component at $0 \in \mathbb A^{k-1}$ (see \cite[p.~11]{stevens2}). In the Artin component, the fibres of this family may be described explicitly by modifying the matrices in the presentation (\ref{presentation1}) to read
\begin{align*}
\begin{pmatrix}
z_0 & z_1 + \lambda_1 & z_2 + \lambda_2 & \cdots & z_{k-1} + \lambda_{k-1}  \\
z_1 & z_2             & z_3             & \cdots & z_k
\end{pmatrix}
\end{align*}
where $(\lambda_1, \dotsc, \lambda_{k-1}) \in \mathbb C^{k-1}$.

The algebras $e_0 A_\lambda e_0$ in Proposition \ref{prop:zkcommutative} correspond to the Artin component in the versal family of commutative deformations of $C = e_0 A e_0$. That is, commutative deformations of $Z_k$ parametrized by $\H^1 (\mathcal T_{Z_k}) \simeq \mathbb C^{k-1}$ as described in \S\ref{section:totalspaces} induce commutative deformations of the $\frac1k(1,1)$ singularity.

\subsubsection{Simultaneous deformations}
\label{subsubsection-simultaneousdeformations}

Recall from Proposition \ref{proposition:varietyzk} the variety $V \subset V_\prec$. Then from Proposition \ref{theorem:singularitypqr} we obtain the following.

\begin{corollary}
Each point $v = (\mu_0, \mu_1, \lambda_1, \dotsc, \lambda_{k-1})$ in the variety $V$ gives an actual deformation $e_0 A_v e_0$ of $e_0 A e_0 \simeq C$, where we denote by $A_v$ the actual deformation of $A$ corresponding to a point $v \in V$. 
\end{corollary}

Note that for $\lambda_1 = \dotsb = \lambda_{k-1} = 0$ the equations defining $V$ are satisfied for any $\mu_0, \mu_1$ so that the points $v = (\mu_0, \mu_1, 0, \dotsc, 0)$ correspond to ``purely noncommutative'' deformations of $X_k$ induced by algebraizable deformation quantizations of algebraic Poisson structures on $Z_k$. Similarly, for $\mu_0 = \mu_1 = 0$ one obtains the irreducible $(k-1)$-dimensional component of ``purely commutative'' deformations of $X_k$. All other points in $V$ may be viewed as ``simultaneously commutative and noncommutative'' deformations of $X_k$.

\begin{remark}
For each $\mu_1 \in \mathbb C$, the point $v = (0, \mu_1, 0, \dotsc, 0)$ corresponds to a ``$q$-deformation'' $e_0 A_v e_0$ of $X_k$. In this case, we have 
\[
e_0 A_v e_0 \simeq \mathbb C \langle z_0, \dotsc, z_{k} \rangle / I
\]
where $I$ is the two-sided ideal generated by 
\begin{flalign*}
&& z_i z_j & -q^{j-i} z_j z_i   && \llap{for $0 \leq i < j \leq k$} \\
&& z_j z_i & -q z_{j+1} z_{i-1} && \llap{for $1 < i \leq j < k$}\\
&& z_j z_1 & -z_{j+1} z_0 && \llap{for $1 \leq j < k$}
\end{flalign*}
where $q = 1 + \mu_1$. (This can be seen by constructing an algebra isomorphism which sends $x_1 y_{k-1}$ to $z_0$ and $x_0 y_{k-i}$ to $z_{i}$ for each $1 \leq i \leq k$.)
It turns out that this family of $q$-deformations leaves the singularity category of the $\frac1k(1,1)$ singularity invariant (cf.\ Remark \ref{remark:singularityinvariant} and Question \ref{question}). The details for the general case of cyclic quotient surface singularities will appear in future work joint with M.~Kalck.
\end{remark}

\subsection*{Acknowledgements}

We would like to thank Tarig Abdelgadir, Pieter Belmans, Johan Commelin, William Crawley-Boevey, Matej Filip, Lutz Hille, Martin Kalck, Bernhard Keller, Travis Schedler, Michel Van den Bergh, Matt Young and Zhiwei Zheng for many helpful comments and discussions.

A large part of this work was carried out while both authors were supported by the Max Planck Institute for Mathematics in Bonn and later by the Hausdorff Research Institute for Mathematics in Bonn as part of the Junior Trimester Program ``New Trends in Representation Theory'', funded by the Deutsche Forschungsgemeinschaft (DFG, German Research Foundation) under Germany's Excellence Strategy -- EXC-2047/1 -- 390685813. We warmly thank both institutes for the welcoming atmosphere and the excellent working environment.

The first named author was also supported by the DFG Research Training Group GK1821 ``Cohomological Methods in Geometry'' at the University of Freiburg. The second named author was also supported by a Humboldt Research Fellowship from the Alexander von Humboldt Foundation and by a National Natural Science Foundation of China (NSFC) grant with number 11871071.

\end{document}